\documentclass[12pt,a4paper]{amsart}
\usepackage[a4paper,inner=2.8cm,outer=2.5cm,top=2.8cm,bottom=2.5cm]{geometry}                
\usepackage{amsmath,amssymb,amscd,amsthm,amsfonts,anyfontsize,color,dsfont,enumerate,fix-cm,layout,lipsum,lpic,mathrsfs,mdwlist,pigpen,stmaryrd,tensor,tikz,thmtools,xspace}
\usepackage[all]{xy}
\usepackage{mathtools}
\mathtoolsset{showonlyrefs}
\usepackage{xfrac}
\theoremstyle{plain}                 
\newtheorem{theorem}{Theorem}[section]     
\newtheorem{proposition}[theorem]{Proposition} 
\newtheorem{corollary}[theorem]{Corollary}     
\newtheorem{lemma}[theorem]{Lemma}        
\theoremstyle{definition}           
\newtheorem{definition}[theorem]{Definition}    
 
\theoremstyle{remark}       
\newtheorem{remark}[theorem]{Remark}    
     
\hypersetup{colorlinks=true,linkcolor=blue,citecolor=purple,filecolor=magenta,urlcolor=cyan}

\newcommand\numberthis{\addtocounter{equation}{1}\tag{\theequation}}

\newcommand{\corc}[1]{\bigg{\langle} \, #1 \,  \bigg{\rangle} ^{\circ}}
\newcommand{\cord}[1]{\bigg{\langle} \, #1 \, \bigg{\rangle} ^{\bullet}}
\def\cE{\mathcal{E}}

\def\cD{\mathcal{D}}

\DeclareRobustCommand{\Stirling}{\genfrac\{\}{0pt}{}}
\DeclareRobustCommand{\stirling}{\genfrac[]{0pt}{}}

\DeclareMathOperator\Sinh{sinh}

\DeclareMathOperator*\Res{Res}
\newcommand{\GL}{\mathord{\mathrm{GL}}}
\DeclareMathOperator{\End}{End}
\DeclareMathOperator{\ad}{ad}

\newcommand{\+}{\! + \!}
\newcommand{\mi}{\! - \!}
\newcommand{\cV}{\mathcal{V}}

\def\Dh{\mathcal{D}^{(h)}}

\def\Z{\mathbb{Z}}

\def\C{\mathbb{C}}

\def\CP1{\mathbb{C}\mathrm{P}^1}

\begin{document}
\title[Polynomiality of monotone orbifold Hurwitz numbers]{Quasi-polynomiality of monotone orbifold Hurwitz numbers and Grothendieck's dessins d'enfants}

\author[R.~Kramer]{R.~Kramer}
\address{R.~K.: Korteweg-de Vries Institute for Mathematics, University of Amsterdam, Postbus 94248, 1090 GE Amsterdam, The Netherlands}
\email{R.Kramer@uva.nl}

\author[D.~Lewanski]{D.~Lewanski}
\address{D.~L.: Korteweg-de Vries Institute for Mathematics, University of Amsterdam, Postbus 94248, 1090 GE Amsterdam, The Netherlands}
\email{D.Lewanski@uva.nl}

\author[S.~Shadrin]{S.~Shadrin}
\address{S.~S.: Korteweg-de Vries Institute for Mathematics, University of Amsterdam, Postbus 94248, 1090 GE Amsterdam, The Netherlands}
\email{S.Shadrin@uva.nl}

\begin{abstract}
We prove quasi-polynomiality for monotone and strictly monotone orbifold Hurwitz numbers. The second enumerative problem is also known as enumeration of a special kind of Grothendieck's dessins d'enfants or $r$-hypermaps. These statements answer positively two conjectures proposed by Do-Karev and Do-Manescu.
We also apply the same method to the usual orbifold Hurwitz numbers and obtain a new proof of the quasi-polynomiality in this case. 

In the second part of the paper we show that the property of quasi-polynomiality is equivalent in all these three cases to the property that the $n$-point generating function has a natural representation on the $n$-th cartesian powers of a certain algebraic curve. These representations are necessary conditions for the Chekhov-Eynard-Orantin topological recursion.  
\end{abstract}

\maketitle

\tableofcontents

\section{Introduction}

This paper is devoted to a combinatorial and analytic study of several kinds of orbifold Hurwitz numbers. The three kinds of orbifold Hurwitz numbers that we consider in this paper are the monotone, the strictly monotone, and the usual ones. Note that the theory of the strictly monotone orbifold Hurwitz numbers is equivalent to the enumeration of hypermaps on two-dimensional surfaces, or, in other words, to the enumeration of some special type of Grothendieck's dessins d'enfants. 

This type of combinatorial objects is important both for purely combinatorial reasons and also because of the numerous relations that these numbers and their generating functions have to the intersection theory of the moduli spaces of curves, matrix models and topological recursion, integrable systems, and low-dimensional topology. We will not make any attempt to survey this very rich theory, and we refer the interested reader to~\cite{ALS, BHLM, BS, DoDyerMathews, DoKarev, DLN, DoManescu, DM, DLPS, GioWa, GGJ1, GGJ2, GGJ3, H, HO, KZ, KZ2, MSS, ZJ, Z} and references therein. 

The Hurwitz numbers of these three types can be efficiently realized as the vacuum expectations in the semi-infinite wedge formalism. These formulae will be the starting point for our paper, and we use them as the definitions of the corresponding Hurwitz numbers. The equivalence with the usual definitions is established via the character formula, and we refer to~\cite{ALS} for that. 

Recall that the Hurwitz numbers that we consider, $h^{\circ,r,\leq}_{g;\vec\mu}$, $h^{\circ,r,<}_{g;\vec\mu}$, and $h^{\circ,r}_{g;\vec\mu}$, depend on a genus parameter $g\geq 0$, and a tuple of $n\geq 1$ positive integers $\vec\mu=(\mu_1,\dots,\mu_n)$. It is a natural combinatorial question how these numbers depend on the parameters $\mu_1,\dots,\mu_n$. We prove in this paper that for $2g-2+n>0$ the dependence on the parameters can be described in a very explicit way. Namely, let us represent any integer $a$ as $r[a]+\langle a \rangle$, $0\leq \langle a\rangle \leq r-1$, and let $\langle \vec\mu \rangle :=(\langle \mu_1 \rangle,\dots,\langle \mu_n \rangle)$. We will use this notation throughout the article. We prove that there exist polynomials $P_{\leq}^\eta$, $P_{<}^\eta$, and $P^\eta$ of degree $3g-3+n$ in $n$ variables, whose coefficients depend on $\eta$ and also on $g$ and $r$, such that 
\begin{align}
h^{\circ,r,\leq}_{g;\vec\mu} & = P_\leq^{\langle \vec\mu \rangle} (\mu_1,\dots,\mu_n) \cdot \prod_{i=1}^n \binom{\mu_i+{[\mu_i]}}{\mu_i};  \\
h^{\circ,r,<}_{g;\vec\mu} & = P_<^{\langle \vec\mu \rangle} (\mu_1,\dots,\mu_n) \cdot \prod_{i=1}^n \binom{\mu_i-1}{[\mu_i]};  \\
h^{\circ,r}_{g;\vec\mu} & = P^{\langle \vec\mu \rangle} (\mu_1,\dots,\mu_n) \cdot \prod_{i=1}^n \frac{\mu_i^{[\mu_i]}}{[\mu_i]!}. 
\end{align}
We call this property quasi-polynomiality. The proof is purely combinatorial and uses some properties of the analogues of the $\mathcal{A}$-operators of Okounkov and Panharipande~\cite{OP} in the semi-infinite wedge formalism. This statement was known for the usual orbifold Hurwitz numbers \cite{BHLM, DLPS, DLN}. In this case we give a new proof. In the cases of monotone and strictly monotone orbifold Hurwitz numbers, this property was conjectured by Do and Karev in~\cite{DoKarev} and Do and Manescu in~\cite{DoManescu}, respectively, and no proof was known.

\subsection{Quasi-polynomiality} Let us explain why the property of being quasi-polynomial is of crucial importance for these Hurwitz numbers, as well as some further results of this paper. For that, we recall several connections of the Hurwitz theory to other areas of mathematics. 

First of all,  there is a connection to the spectral curve topological recursion in the sense of Chekhov-Eynard-Orantin (CEO). This means that the corresponding Hurwitz numbers can be obtained as the coefficients of some particular expansion of the correlation differentials defined on the Cartesian products of some fixed Riemann surface called the spectral curve. These differentials are produced by the CEO topological recursion procedure from a fairly small input data. The input data consists of a curve $\Sigma$, a symmetric bi-differential $B$ defined on $\Sigma\times \Sigma$ with a double pole on the diagonal with biresidue $1$, and two meromorphic functions, $x$ and $y$, defined on $\Sigma$. This allows us to compute recursively the correlation differentials. We need one more piece of data \textemdash\, the variable in which we want to expand the correlation differentials in order to obtain as coefficients the solutions of the combinatorial problem.

In our cases, the data is the following. The curve $\Sigma$ is always $\CP1$ in all three cases. We denote by $z$ a global coordinate on $\CP1$. In the case of $\CP1$ the bi-differential $B(z_1,z_2)$ is uniquely determined by its properties and is equal to $dz_1dz_2/(z_1-z_2)^2$. The functions $x$ and $y$ are the following:
\begin{align}
x& =z(1-z^r), & y & =z^{r-1}/(z^r-1) & \text{in the monotone case;} \\
x& =z^r+z^{-1}, & y & =z & \text{in the strictly monotone case;} \\
x& =\log z - z^r, & y & =z^r & \text{in the usual case.} 
\end{align}
The correlation differentials obtained by the CEO recursion in these cases should be expanded
\begin{align}
\text{in the variable } & x & \text{near } & x=0 & \text{in the monotone case;} \\
\text{in the variable } & x^{-1} & \text{near } & x=\infty & \text{in the strictly monotone case;} \\
\text{in the variable } & e^x & \text{near } & e^x=0 & \text{in the usual case.} 
\end{align}

The topological recursion is proved in the case of the usual orbifold Hurwitz numbers in~\cite{BHLM,DLN}, in the case of strictly monotone Hurwitz numbers it was conjectured in~\cite{DoManescu} and combinatorially proved in~\cite{DOPS}, based on the original derivation of topological recursion in~\cite{CEO} in the case of the two-matrix model. In the case of monotone orbifold Hurwitz numbers only the case $r=1$ has been proved in~\cite{DoDyerMathews}, and a general conjecture was made in~\cite{DoKarev}.

The relation between quasi-polynomiality and the topological recursion is the following. We prove in this paper that a sequence of numbers depending on a tuple $(\mu_1,\dots,\mu_n)$ can be represented as a polynomial in $\mu_1,\dots,\mu_n$ times the non-polynomial factor $\prod_{i=1}^n \binom{\mu_i+{[\mu_i]}}{\mu_i}$ (respectively, $\prod_{i=1}^n \binom{\mu_i-1}{[\mu_i]}$, $\prod_{i=1}^n \mu_i^{[\mu_i]}/[\mu_i]!$) if and only if it can be represented as an expansion of a  special kind of symmetric $n$-differential on the curve $x =z(1-z^r)$ (respectively, 
$x =z^r+z^{-1}$, $x =\log z - z^r$)  in the variable $x$ (respectively, $x^{-1}$, $e^x$).

In the case of the usual orbifold Hurwitz numbers it was already known and used  in~\cite{DLN,BHLM,DLPS}, and, in a slightly different situation, in~\cite{SSZ}. In the case of monotone and strictly monotone orbifold Hurwitz numbers this equivalence was neither explicitly stated nor proved, though it is implicitly suggested in a conjectural form in~\cite{DoKarev} for the monotone and in~\cite{DoManescu} for the strictly monotone cases. Note that since the topological recursion is proved for the strictly monotone Hurwitz numbers independently~\cite{CEO,DOPS}, this equivalence implies the quasi-polynomiality as well.

There are also two unstable cases that have to be studied separately: $(g,n)=(0,1)$ and $(0,2)$. In the case $(g,n)=(0,1)$ (respectively, $(g,n)=(0,2)$) the topological recursion requires that the generating function of the corresponding Hurwitz numbers is given by the expansion of $ydx$ (respectively, $B(z_1,z_2)-B(x_1,x_2)$). For $(g,n)=(0,1)$ this property has been proved in all three cases, in~\cite{DoKarev} for the monotone, in~\cite{DoManescu} for the strictly monotone and in~\cite{DLN,BHLM} for the usual orbifold Hurwitz numbers. Basically, such a representation for the $(g,n)=(0,1)$ generating function is a way to guess a spectral curve for the corresponding combinatorial problem. For $(g,n)=(0,2)$ this property has been proved for strictly monotone and usual orbifold Hurwitz numbers (indeed, the topological recursion is proved in both cases), but it was not known for the monotone case. We prove this in appendix \ref{sec:Unstable}. 

Let us remark that this set of properties (namely, representation of the $(0,1)$ generating function as an expansion of $ydx$, the $(0,2)$ generating function as an expansion of $B(z_1,z_2)-B(x_1,x_2)$, and the quasi-polynomiality property for $2g-2+n>0$) is required for the approach to the topological recursion in~\cite{DMSS}. Once these properties are established, the topological recursion appears to be a Laplace transform of some much easier recursion property of the corresponding combinatorial problem. 

The other important connection for all three Hurwitz theories that we consider here is their relations to the intersection theory of the moduli spaces of curves. It appears that the coefficients of the polynomials in the quasi-polynomial representation of the $n$-point functions can be represented in terms of some intersection numbers on the moduli spaces of curves. This statement is proved for usual Hurwitz numbers for $r=1$ in~\cite{ELSV} and for any $r$ in~\cite{JPT}. 

In general, assume we know that being quasi-polynomial is equivalent to being an expansion of a symmetric differential of certain type. Then in this situation there is an equivalence between the topological recursion and representation in terms of the intersection theory of the moduli spaces of curves. The intersection numbers in this case appear to be the correlators of a certain cohomological field theory, possibly with a non-flat unit. This point of view on topological recursion was first suggested by Eynard in~\cite{Ey} and worked out in detail in many examples, see e.~g.~\cite{DOSS,DKOSS,SSZ,LPSZ}. 

In particular, the cohomological field theory for the case of the strictly monotone orbifold Hurwitz numbers is described in~\cite{DNOPS}. For the monotone orbifold Hurwitz numbers the intersection number formula was derived so far only the case $r=1$, see~\cite{ALS,DoKarev}, and it is based on the proof of the topological recursion in~\cite{DoDyerMathews}.

\subsection{Organization of the paper} In section~\ref{sec:semiinfty} we briefly recall the necessary background on semi-infinite wedge formalism. In section~\ref{sec:symmstirling} we review the interplay between symmetric polynomials and Stirling numbers, together with their generating function. In section~\ref{sec:Aoperators} we define the $\mathcal{A}$-operators and we express the generating series for monotone and strictly monotone Hurwitz numbers in terms of $\mathcal{A}$-operators acting on the Fock space. The main result of the paper is stated and proved in section~\ref{sec:Poly}. In section~\ref{sec:Correlation} the polynomiality properties are proved to be equivalent to the analytic properties that are necessary for the Chekhov-Eynard-Orantin topological recursion. 
Finally, in appendix~\ref{sec:Unstable} we perform the computations for the unstable $(0,1)$, as an example of the usage of the $\mathcal{A}$-operators, and we prove a formula relating the $(0,2)$-generating function for the monotone orbifold Hurwitz numbers to the expansion of the Bergman kernel.

\subsection{Acknowledgments} We would like to thank A.~Alexandrov, N.~Do, P.~Dunin-Barkowski, M.~Karev, and A.~Popolitov for interesting discussions and very useful remarks. The authors are supported by the Netherlands Organization for Scientific Research.

\section{Semi-infinite wedge formalism}\label{sec:semiinfty}

In this section we briefly recall the semi-infinite wedge formalism. It is nowadays a standard tool in Hurwitz theory, with many good introductions to it. We refer the reader, for instance, to~\cite{Joh} and~\cite{ALS} and references therein for a more complete exposition.

Let $V$ be an infinite-dimensional complex vector space with a basis labeled by half-integers. Denote the basis vector labeled by $m/2$ by $\underline{m/2}$, so $V = \bigoplus_{i \in \Z + \frac{1}{2}} \C \underline{i}$.

\begin{definition}\label{DEFSEMIINF}
The semi-infinite wedge space $\bigwedge^{\frac{\infty}{2}}(V) = \mathcal{V}$ is defined to be the span of all of the semi-infinite wedge products of the form
\begin{equation}\label{wedgeProduct}
\underline{i_1} \wedge \underline{i_2} \wedge \cdots
\end{equation}
for any decreasing sequence of half-integers $(i_k)$ such that there is an integer $c$ with $i_k + k - \frac{1}{2} = c$ for $k$ sufficiently large. The constant $c$ is called the \textit{charge}. We give $\mathcal{V}$ an inner product $(\cdot,\cdot)$ declaring its basis elements to be orthonormal.
\end{definition}

\begin{remark}
By definition \ref{DEFSEMIINF} the charge-zero subspace $\mathcal{V}_0$ of $\mathcal{V}$ is spanned by semi-infinite wedge products of the form 
$$
\underline{\lambda_1 - \frac{1}{2}} \wedge \underline{\lambda_2 - \frac{3}{2}} \wedge \cdots
$$
for some integer partition~$\lambda$. Hence we can identify integer partitions with the basis of this space:
\begin{equation*}
\mathcal{V}_0 = \bigoplus_{n \in \mathbb{N} } \bigoplus_{\lambda\,  \vdash\, n} \C v_{\lambda}
\end{equation*}
\end{remark}

The empty partition $\emptyset$ plays a special role.
We call $$v_{\emptyset} = \underline{-\frac{1}{2}} \wedge \underline{-\frac{3}{2}} \wedge \cdots$$ the vacuum vector and we denote it by $|0\rangle$. Similarly we call the covacuum vector its dual with respect to the scalar product $(\cdot,\cdot)$ and we denote it by $\langle0|$.

\begin{definition}
 The \emph{vacuum expectation value} or \emph{disconnected correlator} $\langle \mathcal{P}\rangle^{\bullet}$ of an operator~$\mathcal{P}$ acting on $\mathcal{V}_0$ is defined to be:
\begin{equation*}
\langle \mathcal{P}\rangle^{\bullet} \coloneqq (|0\rangle , \mathcal{P} |0\rangle) \eqqcolon \langle 0 |\mathcal{P}|0 \rangle
\end{equation*}
We also define the functions
\begin{equation}
\zeta(z)=e^{z/2} - e^{-z/2} = 2 \Sinh(z/2)
\end{equation}
and 
\begin{equation}
\mathcal{S}(z) = \frac{\zeta(z)}{z} = \frac{\Sinh(z/2)}{z/2}.
\end{equation}
\end{definition}

\begin{definition} This is the list of operators we will use:
\begin{enumerate}
\item[i)] For $k$ half-integer the operator
$\psi_k \colon (\underline{i_1} \wedge \underline{i_2} \wedge \cdots) \ \mapsto \ (\underline{k} \wedge \underline{i_1} \wedge \underline{i_2} \wedge \cdots)$
increases the charge by $1$. Its adjoint operator $\psi_k^*$ with respect to~$(\cdot,\cdot)$
 decreases the charge by $1$.
 \item[ii)] The normally ordered products of $\psi$-operators
\begin{equation}
E_{i,j} \coloneqq \begin{cases}\psi_i \psi_j^*, & \text{ if } j > 0 \\
-\psi_j^* \psi_i & \text{ if } j < 0\ .\end{cases} 
\end{equation}
preserve the charge and hence can be restricted to $\cV_0$ with the following action. For $i\neq j$, $E_{i,j}$ checks if $v_\lambda$ contains $\underline{j}$ as a wedge factor and if so replaces it by $\underline{i}$. Otherwise it yields~$0$. In the case $i=j > 0$, we have $E_{i,j}(v_\lambda) = v_\lambda$ if $v_\lambda$ contains $\underline{j}$ and $0$ if it does not; in the case  $i=j < 0$, we have $E_{i,j}(v_\lambda) = - v_\lambda$ if $v_\lambda$ does not contain $\underline{j}$ and $0$ if it does. This gives a projective representation of \( \mathcal{A}_\infty \), the Lie algebra of complex \( \Z \times \Z \) matrices with only finitely many non-zero diagonals~\cite{Joh}. 
 \item[iii)] The diagonal operators are assembled into the operators
\begin{equation}
\mathcal{F}_n \coloneqq \sum_{k\in\Z+\frac12} \frac{k^n}{n!} E_{k,k} 
\end{equation}
The operator $C \coloneqq \mathcal{F}_0$ is called \emph{charge operator}, while the operator $E \coloneqq \mathcal{F}_1$ is called \emph{energy operator}. Note that $\mathcal{F}_0$ vanishes identically on $\cV_0$.
We say that an operator $\mathcal{P}$ on $\cV_0$ is of energy $c\in\mathbb{Z}$ if
\begin{equation}
[\mathcal{P}, E] = c\, \mathcal{P}
\end{equation}
The operator $E_{i,j}$ has energy $j-i$, hence all the $\mathcal{F}_n$'s have zero energy. Operators with positive energy annihilate the vacuum while negative energy operators are annihilated by the covacuum.
\item[iv)] For $n$ any integer and $z$ a formal variable one has the energy $n$ operators:
\begin{equation}
\mathcal{E}_n(z) = \sum_{k \in \Z + \frac12} e^{z(k - \frac{n}{2})} E_{k-n,k} + \frac{\delta_{n,0}}{\zeta(z)}  .
\end{equation}
\item[v)] For $n$ any nonzero integer one has the energy $n$ operators:
\begin{equation}
\alpha_n = \mathcal{E}_n(0) = \sum_{k \in \Z + \frac12} E_{k-n,k}
\end{equation}
These \( \alpha_n \) can also be interpreted as elements of \( \GL (V)\), in which case \( \alpha_n^{-1} = \alpha_{-n} \).
\end{enumerate}
\end{definition}

The commutation formula for $\mathcal{E}$ operators reads:
\begin{equation}\label{eq:commE}
  \left[\cE_a(z),\cE_b(w)\right] =
\zeta\left(\det  \left[
\begin{matrix}
  a & z \\
b & w
\end{matrix}\right]\right)
\,
\cE_{a+b}(z+w)
\end{equation}
and in particular \( [\alpha_k, \alpha_l ] = k \delta_{k+l,0}\).\par
Note that $\cE_k(z)\big|0\big\rangle = 0$ if $k > 0$, while $\cE_0(z)\big|0\big\rangle = {\zeta(z)}^{-1}\big|0\big\rangle$. We will also use the $\cE$ operator without the correction in energy zero, i.e.
\begin{equation*}
\tilde{\mathcal{E}_0}(z) = \sum_{k \in \Z + \frac12} e^{zk} E_{k,k} = \sum_{n=0}^{\infty} \mathcal{F}_n z^n = C + Ez + \mathcal{F}_2 z^2 + \dots
\end{equation*}
which annihilates the vacuum and obeys the same commutation rule as $\mathcal{E}_0$.

\section{Symmetric polynomials and Stirling numbers}\label{sec:symmstirling}

In this section we recollect some combinatorial notions used in the rest of the paper. In particular we recall here some basic facts on homogeneous symmetric polynomials and Stirling numbers, and their interconnection.

\subsection{Symmetric polynomials}

\begin{definition}
Let \( X = \{ x_1, \dotsc, x_n\} \) be a finite set of variables. The \emph{complete symmetric polynomials \( h_k\)} and the \emph{elementary symmetric polynomials \( \sigma_k \)} on \( X \) are defined as follows:
\begin{align}
h_k (X) &=\sum_{1 \leq i_1 \leq i_2 \leq \dotsb \leq i_k \leq n} x_{i_1} \dotsb x_{i_k} \\
\sigma_k (X) &=\sum_{1 \leq i_1 < i_2 < \dotsb < i_k \leq n} x_{i_1} \dotsb x_{i_k}
\end{align}
\end{definition}
The properties of these functions are well-documented, see e.g. \cite{McD}. We will list some useful properties.
\begin{lemma}
The generating functions of the complete and elementary symmetric polynomials are as follows:
\begin{align}
\sum_{k=0}^\infty h_k (x_1, \dotsc, x_n) u^k &= \prod_{i=1}^n \frac{1}{1-ux_i}\\
\sum_{k=0}^\infty \sigma_k (x_1, \dotsc, x_n) u^k &= \prod_{i=1}^n (1+ux_i)
\end{align}
\end{lemma}
\begin{corollary}
For any finite set of variables \( X \),
\begin{equation}\label{eq:sympolrel}
\sum_{k=0}^\infty h_k (X ) u^k \sum_{l=0}^\infty \sigma_l (X) (-u)^l = 1
\end{equation}
\end{corollary}
The following lemma is an easy consequence of the definitions, and can be proved by induction on the number of arguments.
\begin{lemma}
If the variables in a symmetric polynomial are all offset by the same amount, they can be re-expressed as a linear combination of non-offset symmetric polynomials as follows:
\begin{equation}\label{eq:hsympoloffset}
h_k(x_1 + A, \dotsc, x_n + A) = \sum_{i=0}^{k}\binom{k + n - 1}{i} h_{k-i}(x_1, \dotsc, x_n) A^{i}
\end{equation}
\begin{equation}\label{eq:ssympoloffset}
\sigma_k(x_1 + A, \dotsc, x_n + A) = \sum_{i=0}^{k}\binom{n+ i - k}{i} \sigma_{k-i}(x_1, \dotsc, x_n) A^{i}
\end{equation}
\end{lemma}

\subsection{Stirling numbers}

We now recall some notions on Stirling numbers. A complete treatment of the subject can be found in \cite{CHA}.
\begin{definition}
The \textit{(unsigned) Stirling numbers of the first kind} $\stirling{i}{t}$  are defined as coefficients of the following expansion in the formal variable $T$
\begin{equation} (T)_i = \sum_{t=0}^i \stirling{i}{t} T^t
\end{equation}
where $i, t$ are nonnegative integers and the subscript indicates the {\em Pochhammer symbol}:
\begin{equation*} 
(x+1)_n=\frac{(x+n)!}{x!}=\left\{\begin{array}{ll} (x+1)(x+2)\cdots(x+n) &n\geq 0 \\
(x(x-1)\cdots(x+n+1))^{-1} &n\leq 0\end{array} \right. .
\end{equation*}
From the definition, $(x+1)_n$ vanishes for integers $x$ satisfying $-n\leq x\leq -1$, and $1/(x+1)_n$ vanishes for integers $x$ satisfying $0\leq x \leq -(n+1)$.
\par
The \textit{Stirling numbers of the second kind} $\Stirling{i}{t}$  are defined as coefficients of the following expansion in the formal variable $T$
\begin{equation} T^i = \sum_{t=0}^i \Stirling{i}{t} (T-t+1)_t
\end{equation}
where $i, t$ are nonnegative integers. Note that for $t > i$ we have $\stirling{i}{t} = \Stirling{i}{t} = 0$.
\end{definition}
The complete and elementary polynomials evaluated at integers are linked to the Stirling numbers by the following relation. 
\begin{align}
\sigma_v(1, 2, \dots, t-1) &= \stirling{t}{t-v}\label{eq:scomeS}\\
h_v(1, 2, \dots, t) &= \Stirling{t \+ v}{t} \label{eq:hcomeS}
\end{align}
The expressions in terms of generating series read 
\begin{lemma}\label{lem:genser} We have:
\begin{align}
\stirling{j}{t} &= [y^{j-t}]. \frac{(j-1)!}{(t-1)!} \mathcal{S}(y)^{-j} e^{yj/2}; & \Stirling{j}{t} &= [y^{j-t}]. \frac{j!}{t!} \mathcal{S}(y)^t e^{yt/2}.
\end{align}
\end{lemma}

\section{\texorpdfstring{$ \mathcal{A}$}{A}-operators for monotone orbifold Hurwitz numbers} \label{sec:Aoperators}
In this section we express the generating series for monotone and strictly monotone orbifold Hurwitz numbers in terms of correlators of certain $\mathcal{A}$-operators acting on the Fock space.

\subsection{Generating series for monotone orbifold Hurwitz numbers}

Let us define the genus-generating series for disconnected monotone and strictly monotone orbifold Hurwitz numbers as
\begin{equation}\label{eq:GenSer}
H^{\bullet, r, \leq}(u,\vec\mu) \coloneqq \sum_{g = 0}^{\infty} \Big( h^{r,\leq}_{g; \vec\mu} \Big) u^b, \qquad \qquad  H^{\bullet, r, <}(u,\vec\mu) \coloneqq \sum_{g = 0}^{\infty} \Big( h^{r,<}_{g; \vec\mu} \Big) u^b
\end{equation}
 where, by Riemann-Hurwitz, $b$ is the number of simple ramifications 
 $$b = 2g - 2 + l(\mu) + |\mu|/r.$$
We want to express the generating series through the semi-infinite wedge formalism. In \cite{ALS} it was proved that the eigenvalue of the operator 
$$ \cD^{(h)}(u) 
 \coloneqq \exp\left(\left[\frac{\tilde \cE_0\left(u^2\frac{d}{du}\right)}{\zeta\left(u^2\frac{d}{du}\right)} - E\right] .\log u \right) \\$$
acting on the basis of the charge zero sector of the Fock space is the generating series for the complete symmetric polynomials, in the sense that
$$\mathcal{D}^{(h)}(u) .v_\lambda  =  \sum_{k=0}^\infty h_k(\textsf{cr}^{\lambda})u^{k}  v_{\lambda},$$
where the set of variables $\textsf{cr}^{\lambda}$ is the content of Young tableau $\lambda$. Similarly, the operator
$$ \cD^{(\sigma)}(u) 
 \coloneqq \exp\left( - \left[\frac{\tilde \cE_0\left(- u^2\frac{d}{du}\right)}{\zeta\left(- u^2\frac{d}{du}\right)} - E\right] .\log u \right) \\$$
produces as eigenvalue the generating series for elementary symmetric polynomials:
$$\mathcal{D}^{(\sigma)}(u) .v_\lambda  =  \sum_{k=0}^\infty \sigma_k(\textsf{cr}^{\lambda})u^{k}  v_{\lambda}.$$
The generating series in equation~\eqref{eq:GenSer} therefore read respectively
\begin{equation}\label{eq:GenSerFock}
H^{\bullet, r, \leq}(u,\vec\mu) = \cord{ e^{\frac{\alpha_{r}}{r}} \mathcal{D}^{(h)}(u) \prod_{i=1}^n \frac{\alpha_{-\mu_i}}{\mu_i} } 
\end{equation}
and
\begin{equation}\label{eq:GenSerFocksigma}
 H^{\bullet, r, <}(u,\vec\mu) = \cord{ e^{\frac{\alpha_{r}}{r}} \mathcal{D}^{(\sigma)}(u) \prod_{i=1}^n \frac{\alpha_{-\mu_i}}{\mu_i} } 
\end{equation}

\subsection{Conjugations of operators}
In this section we prove several lemmata that we will use later.
\begin{lemma}\label{lem: firstconj} We have:
\begin{align*}
\mathcal{O}^h_\mu (u ) &\coloneqq
  \mathcal{D}^{(h)}(u) \alpha_{-\mu} \mathcal{D}^{(h)}(u)^{-1} = \sum_{k \in \Z + \frac12} \sum_{v=0}^{\infty} h_v(1 + k - \sfrac{1}{2}, \dots, \mu + k - \sfrac{1}{2}) u^v E_{k + \mu, k} ;
\\
\mathcal{O}^{\sigma}_\mu (u ) &\coloneqq
  \mathcal{D}^{(\sigma)}(u) \alpha_{-\mu} \mathcal{D}^{(\sigma)}(u)^{-1} = \sum_{k \in \Z + \frac12} \sum_{v=0}^{\infty} \sigma_v(1 + k - \sfrac{1}{2}, \dots, \mu + k - \sfrac{1}{2}) u^v E_{k + \mu, k} .
\end{align*}
\end{lemma}

\begin{proof}
We prove only the first equation, since the proof for the second is completely analogous.
Applying the change of variable $u(z) = -z^{-1}$, we have 
\begin{equation}
\cD^{(h)}(u(z)) 
 = \exp\left(- \frac{\tilde \cE_0\left(\frac{d}{dz}\right)}{\zeta\left(\frac{d}{dz}\right)}.\log (-z) \right) (-z)^E  =: e^{B(z)} (-z)^E
\end{equation}
Observe that the operator $B(z)$ has zero energy and hence commutes with $(-z)^E$. On the other hand, the operator $\alpha_{-\mu}$ has energy $-\mu$, hence the conjugation by the operator $(-z)^E$ produces the extra factor $(-z)^{\mu}$. By the Hadamard lemma we can expand the conjugation as
\begin{equation} 
\mathcal{D}^{(h)}(u) \alpha_{-\mu} \mathcal{D}^{(h)}(u)^{-1} = (-z)^{\mu} \sum_{s=0}^{\infty} \frac{1}{s!}\ad_{B(z)}^s (\alpha_{-\mu})
\end{equation}
It is enough to show that 
\begin{equation}\label{eq: commutator}
\ad_{B(z)}^s (\alpha_{-\mu}) = \sum_{k \in \Z + \frac12} {\log\left(\prod_{l=0}^{\mu - 1} \frac{1}{(-z -l - k - \sfrac{1}{2})} \right)}^s E_{k + \mu, k}
\end{equation}
Indeed this would imply 
\begin{equation*} 
\mathcal{D}^{(h)}(u) \alpha_{-\mu} \mathcal{D}^{(h)}(u)^{-1} = \sum_{k \in \Z + \frac12} \left(\prod_{l=0}^{\mu - 1} \frac{1}{1- (l + k + \sfrac{1}{2})(-z^{-1})} \right) E_{k + \mu, k}
\end{equation*}
which proves the lemma by substituting back $u = -z^{-1}$ and expanding in the generating series for complete symmetric polynomials. Let $C(s)$ be the left hand side of equation~\eqref{eq: commutator}. We compute:
\begin{align*}
C(s) &= \left[- \frac{\tilde \cE_0\left(\frac{d}{dz_s}\right)}{\zeta\left(\frac{d}{dz_s}\right)}, \dots \left[ - \frac{\tilde \cE_0\left(\frac{d}{dz_1}\right)}{\zeta\left(\frac{d}{dz_1}\right)}, \mathcal{E}_{-\mu}(0)\right] \dots \right].\prod_{i=1}^s \log(-z_i) \Big{|}_{z_i = z} \\
& = (-1)^s \prod_{i=1}^s \frac{ \zeta\left(\mu \frac{d}{dz_i}\right)}{\zeta\left(\frac{d}{dz_i}\right)}  \mathcal{E}_{-\mu}\left(\sum_{i=1}^s \frac{d}{dz_i}\right).\prod_{i=1}^s \log(-z_i) \Big{|}_{z_i = z} \\
& =  \sum_{k \in \Z + \sfrac{1}{2}} \prod_{i=1}^s  \sum_{l=0}^{\infty} -\left(e^{\frac{d}{dz_i}(\mu + k - l -\sfrac{1}{2})} - e^{\frac{d}{dz_i}(k - l -\sfrac{1}{2} )}\right).\log(-z_i) E_{k + \mu, k}\Big{|}_{z_i = z}
\end{align*}
Observe that the summation over $l$ is the result of the expansion in geometric formal power series of $1/(1 - e^{-d/dz_i})$.
The expression in the last line equals the right hand side of equation \eqref{eq: commutator} since the $s$ operators act independently, and using \( e^{a\frac{d}{dz}} f(z) = f(z+a)\). The lemma is proved.
\end{proof}

In the following lemma, we calculate the inverse of the  \( \mathcal{O}\)-operators, defined in lemma~\ref{lem: firstconj}, when viewed as operators on the space \( V\). 

\begin{lemma} The \( \mathcal{O}\)-operators can be viewed as elements of the ring \( \End (V)\llbracket u\rrbracket \), and considered as such are invertible with the following inverses:
\begin{align}
{\mathcal{O}^{h}_\mu (u )}^{-1}  &= \sum_{k \in \Z + \frac12} \sum_{v=0}^{\infty} \sigma_v(1 + k - \sfrac{1}{2}, \dots, \mu + k - \sfrac{1}{2}) (-u)^v E_{k, k - \mu} \\
{\mathcal{O}^{\sigma}_\mu (u )}^{-1}  &= \sum_{k \in \Z + \frac12} \sum_{v=0}^{\infty} h_v(1 + k - \sfrac{1}{2}, \dots, \mu + k - \sfrac{1}{2}) (-u)^v E_{k, k - \mu} 
\end{align}
\begin{proof}
This follows from the duality between generating series of complete and elementary symmetric polynomials expressed in equation~\eqref{eq:sympolrel}, and the form of the \( \mathcal{O}\)-operators in lemma~\ref{lem: firstconj}.
\end{proof}
\end{lemma}
Because of the way we constructed the \( \mathcal{O}\)-operators, we have that
\begin{align}
{\mathcal{O}^{h}_\mu (u )}^{-1} &= \mathcal{D}^{(h)}(u) \alpha_\mu \mathcal{D}^{(h)}(u)^{-1} & {\mathcal{O}^{\sigma}_\mu (u )}^{-1} &= \mathcal{D}^{(\sigma)}(u) \alpha_\mu \mathcal{D}^{(\sigma)}(u)^{-1}
\end{align}
From now on, we will keep using this notation also if we consider actions of these operators on \( \mathcal{V} \).

\begin{corollary}
The different kinds of \( \mathcal{O}\)-operators can also be written as follows:
\begin{align}
\mathcal{O}^h_\mu (u) &= \sum_{v = 0}^{\infty} \frac{ (v \+ \mu \mi 1)!}{(\mu \mi 1) !}[z^v] \mathcal{S}(uz)^{\mu - 1} \mathcal{E}_{- \mu}( uz)\\
{\mathcal{O}^h_\mu (u)}^{-1} &= \sum_{v=0}^\mu \frac{\mu!}{(\mu \mi v)!} [z^v] \mathcal{S}(uz)^{-\mu -1} \mathcal{E}_{\mu} (-uz) \\
\mathcal{O}^{\sigma}_\mu (u) &= \sum_{v=0}^\mu \frac{\mu!}{(\mu \mi v)!} [z^v] \mathcal{S}(uz)^{-\mu -1} \mathcal{E}_{-\mu} (uz)  \\
{\mathcal{O}^{\sigma}_\mu (u)}^{-1} &= \sum_{v = 0}^{\infty} \frac{ (v \+ \mu \mi 1)!}{(\mu \mi 1) !} [z^v] \mathcal{S}(uz)^{\mu - 1} \mathcal{E}_{\mu}( -uz) 
\end{align}
\end{corollary}
\begin{proof}
We will first derive the first equation, starting from lemma~\ref{lem: firstconj}. First we use equation~\eqref{eq:hsympoloffset}:
\begin{align}
\mathcal{O}^h_\mu (u ) &= \sum_{k \in \Z + \frac12} \sum_{v=0}^{\infty} h_v(1 + k - \sfrac{1}{2}, \dots, \mu + k - \sfrac{1}{2}) u^v E_{k + \mu, k}\\
&= \sum_{k \in \Z + \frac12} \sum_{v=0}^{\infty} \sum_{i=0}^v \binom{v \+ \mu \mi 1}{i} h_{v-i}(0, \dots, \mu -1) \big( k + \frac{1}{2} \big)^i u^v E_{k + \mu, k}
\end{align}
By equation~\refeq{eq:hcomeS} and lemma~\ref{lem:genser}, we then get:
\begin{align}
\mathcal{O}^h_\mu (u ) &= \sum_{k \in \Z + \frac12} \sum_{v=0}^{\infty} \sum_{i=0}^v \binom{v \+ \mu \mi 1}{i}[y^{v-i}] \frac{(v\+ \mu \mi i \mi 1)!}{(\mu \mi 1)!}\mathcal{S}(y)^{\mu - 1} e^{y\frac{\mu - 1}{2}} [z^i]i! e^{z(k + \frac{1}{2})} u^v E_{k + \mu, k}\\
&= \sum_{v=0}^{\infty} \frac{(v \+ \mu \mi 1)!}{(\mu \mi 1)!}[z^v] \mathcal{S}(uz)^{\mu - 1} \mathcal{E}_{-\mu}(uz)
\end{align}
For the other equations, the calculation is similar, replacing the equations for the complete symmetric polynomials with their counterparts for the elementary symmetric polynomials where necessary.
\end{proof}

\begin{lemma}\label{conjall}
\begin{align}
\qquad \qquad e^{\frac{\alpha_{r}}{r}} \mathcal{O}^h_\mu (u)  e^{-\frac{\alpha_{r}}{r}} &= \sum_{t=0}^{\infty} \sum_{v = t}^{\infty} \frac{ (v + \mu - 1)!}{t! \,(\mu-1) !} u^t
[z^{v-t}] \mathcal{S}(uz)^{\mu - 1} \mathcal{S}(ruz)^t \mathcal{E}_{tr - \mu}( uz)  \numberthis \label{eq:conj1}
\\
e^{\frac{\alpha_{r}}{r}} {\mathcal{O}^h_\mu (u)}^{-1}  e^{-\frac{\alpha_{r}}{r}} &= \sum_{t=0}^\mu  \sum_{v=t}^\mu \frac{\mu!}{t!(\mu \mi v)!} (-u)^t [z^{v-t}] \mathcal{S}(uz)^{-\mu -1} \mathcal{S}(ruz)^t \mathcal{E}_{tr+\mu} (-uz)  \label{eq:conj2}
\\
e^{\frac{\alpha_{r}}{r}} \mathcal{O}^{\sigma}_\mu (u) e^{-\frac{\alpha_{r}}{r}} &= \sum_{t=0}^\mu  \sum_{v=t}^\mu \frac{\mu!}{t!(\mu \mi v)!} u^t [z^{v-t}] \mathcal{S}(uz)^{-\mu -1} \mathcal{S}(ruz)^t \mathcal{E}_{tr-\mu} (uz)  \label{eq:conj3}
\\
e^{\frac{\alpha_{r}}{r}} {\mathcal{O}^{\sigma}_\mu (u)}^{-1}  e^{-\frac{\alpha_{r}}{r}} &= \sum_{t=0}^{\infty} \sum_{v = t}^{\infty} \frac{ (v + \mu - 1)!}{t! \,(\mu-1) !} (-u)^t
[z^{v-t}] \mathcal{S}(uz)^{\mu - 1} \mathcal{S}(ruz)^t \mathcal{E}_{tr + \mu}( -uz) \numberthis \label{eq:conj4}
\end{align}
\end{lemma}

\begin{proof} Let us prove equation \eqref{eq:conj1}. Applying  the Hadamard lemma as in lemma~\ref{lem: firstconj} we find
\begin{align}
 e^{\frac{\alpha_{r}}{r}} \mathcal{O}_\mu (u) e^{-\frac{\alpha_{r}}{r}} &= \sum_{t=0}^\infty \frac{1}{ t! r^t} \ad_{\alpha_r}^t \Big( \sum_{v = 0}^{\infty} \frac{ (v + \mu - 1)!}{(\mu-1) !}[z^v] \mathcal{S}(uz)^{\mu - 1} \mathcal{E}_{- \mu}( uz) \Big)\\
&= \sum_{t=0}^\infty \sum_{v = 0}^{\infty} \frac{ (v + \mu - 1)!}{t! \, (\mu-1)! r^t}[z^v] \mathcal{S}(uz)^{\mu - 1} \ad_{\alpha_r}^t  \mathcal{E}_{- \mu}( uz)
\end{align}
By equation~\eqref{eq:commE}, we know
\begin{equation}
\ad_{\alpha_r} \mathcal{E}_{-\mu}(uz) = \zeta (ruz)\mathcal{E}_{r-\mu}(uz)
\end{equation}
Using this \( t\) times, we get that
\begin{align}
 e^{\frac{\alpha_{r}}{r}} \mathcal{O}_\mu (u) e^{-\frac{\alpha_{r}}{r}} &= \sum_{t=0}^\infty \sum_{v = 0}^{\infty} \frac{ (v + \mu - 1)!}{t! \, (\mu-1)! r^t}[z^v] \mathcal{S}(uz)^{\mu - 1} \zeta (ruz)^t  \mathcal{E}_{tr- \mu}( uz) \\
 &= \sum_{t=0}^\infty \sum_{v = 0}^{\infty} \frac{ (v + \mu - 1)!}{t! \, (\mu-1)!}u^t[z^{v-t}] \mathcal{S}(uz)^{\mu - 1} \mathcal{S} (ruz)^t  \mathcal{E}_{tr- \mu}( uz)
\end{align}
For the other equations, the calculation is completely analogous, using that \( \mathcal{S} \) is an even function.
This finishes the proof of the lemma.
\end{proof}

\subsection{\texorpdfstring{$ \mathcal{A}$}{A}-operators}

Let us now define the $\mathcal{A}$-operators for the $r$-orbifold monotone Hurwitz numbers as
\begin{align}\label{eq:FormulaForA-Operator}
& \mathcal{A}^{h}_{\langle \mu \rangle}(u,\mu) =  \sum_{t \in \Z} \sum_{v=t}^{\infty} \frac{ ([\mu ]  \! + \! \mu \! + \! 1)_{v-1}}{([\mu ] + 1)_t } [z^{v-t}] \mathcal{S}(uz)^{\mu - 1} \mathcal{S}(ruz)^{t+[\mu]} \mathcal{E}_{tr - \langle \mu \rangle}( uz)\\
\label{eq:FormulaForA-Operatorsigma}
& \mathcal{A}^{\sigma}_{\langle \mu \rangle}(u,\mu) =  \sum_{t = - \infty}^{\mu - [\mu ]} \sum_{v=t}^{\mu - [\mu ]} \frac{ ( \mu \mi [\mu] \mi v \+ 1)_{v-1}}{([\mu ] + 1)_t } [z^{v-t}] \mathcal{S}(uz)^{- \mu - 1} \mathcal{S}(r uz)^{t+[\mu]} \mathcal{E}_{tr - \langle \mu \rangle}(uz)
\end{align}
where \( \mu = r[\mu ] + \langle \mu \rangle \) denotes the euclidean division by $r$.

\begin{proposition}\label{AhOper}
\begin{align}
H^{\bullet, r, \leq}(u,\vec{\mu} ) &= u^{\frac{d}{r}} \prod_{i=1}^{l(\vec\mu )} \binom{\mu_i  + [\mu_i ]  }{\mu_i} \cord{ \prod_{i=1}^{l(\vec\mu)} \mathcal{A}^h_{\langle \mu_i \rangle }(u, \mu_i) } \label{eq:FormulaForGeneratingFunction} \\
H^{\bullet, r, <}(u,\vec{\mu} ) &= u^{\frac{d}{r}}  \prod_{i=1}^{l(\vec\mu )} 
\binom{\mu_i - 1}{[\mu_i]}
\cord{ \prod_{i=1}^{l(\vec\mu)} \mathcal{A}^{\sigma}_{\langle \mu_i \rangle }(u, \mu_i) } \label{eq:FormulaForGeneratingFunctionsigma}
\end{align}
where \( \mu = r[\mu ] + \langle \mu \rangle \) denotes the euclidean division by $r$.
\end{proposition}
\begin{proof} Let us prove equation \eqref{eq:FormulaForGeneratingFunction}. Observe that both the operators $\tilde{\mathcal{E}}$ and $\alpha_r$ annihilate the vacuum. Hence inserting the operators $\Dh $ and $e^{\alpha_r}$ acting on the vacuum does not change the expression in equation~\eqref{eq:GenSerFock}:
\begin{equation}
H^{\bullet, r, \leq}(u,\vec\mu) = \cord{ \prod_{i=1}^n \frac{1}{\mu_i}e^{\frac{\alpha_{r}}{r}} \mathcal{D}^{(h)}(u)  \alpha_{-\mu_i} (\mathcal{D}^{(h)}(u))^{-1} e^{\frac{-\alpha_{r}}{r}}}
\end{equation}
The operators in the correlator are given by formula~\eqref{eq:conj1}, divided by \( \mu \). For every $i = 1, \dots, n$, rescale the $t$-sum in formula~\eqref{eq:conj1} by $ t_\text{new} \coloneqq t - [\mu_i]$ and the \( v\)-sum by \( v_\text{new} \coloneqq v - [\mu_i ]\), and conjugate by the operator $u^{\mathcal{F}_1/r}$. The latter operation has the effect of annihilating the factor $u^t$ and of creating a factor $u^{\mu_i /r} $ that can be written outside the sum. Extracting the binomial coefficient in equation~\eqref{eq:FormulaForGeneratingFunction} and extending the $t$-sum over all integers (since the Pochhammer symbol in the denominator is infinite for \( t < - [\mu_i] \)) proves equation~\eqref{eq:FormulaForGeneratingFunction}.\par
 The proof for equation~\eqref{eq:FormulaForGeneratingFunctionsigma} is analogous, starting from the operator given by formula~\eqref{eq:conj3}. After rescaling the $t$- and \( v\)-sums and conjugating with $u^{\mathcal{F}_1/r}$, we extract from the correlator the factor
\begin{equation}
\frac{(\mu - 1)!}{[\mu]!(\mu - [\mu] - 1 )!}
\end{equation}
Here, we can also extend the sum to \( + \infty \), because the Pochhammer symbol in the numerator is zero for the added terms.
Proposition~\ref{AhOper} is proved.
\end{proof}

%
%

\begin{proposition}
The inverses of the \( \mathcal{A}\)-operators (as elements of \( \End (V)\llbracket u \rrbracket \)) are given as follows:
\begin{align}\label{eq:A-inv}
\mathcal{A}^h_{\langle \mu \rangle}(u,\mu )^{-1} &= \sum_{t=0}^\mu  \sum_{v=t}^\mu \frac{(-1)^t(\mu \+ [\mu ])! \mu}{t! (\mu \mi v)![\mu ]!} [z^{v-t}] \mathcal{S}(uz)^{-\mu -1} \mathcal{S}(ruz)^t \mathcal{E}_{tr+\mu} (-uz)\\
\label{eq:A-invsigma}
\mathcal{A}^{\sigma}_{\langle \mu \rangle}(u,\mu )^{-1} &= \sum_{t=0}^\infty  \sum_{v=t}^\infty \frac{(-1)^t(v \+ [\mu ] -1)! \mu}{t! (\mu \mi [\mu] -1)![\mu ]! } [z^{v-t}] \mathcal{S}(uz)^{\mu -1} \mathcal{S}(ruz)^t \mathcal{E}_{tr+\mu} (-uz)
\end{align}
\end{proposition}
\begin{proof}
Let us prove equation \eqref{eq:A-inv}. By lemma \ref{conjall} and proposition \ref{AhOper},  the inverse operator is given by
\begin{equation}
\mathcal{A}^h_{\langle \mu \rangle}(u,\mu )^{-1} = u^{\mu/r} \mu\binom{\mu + [\mu ]}{\mu} u^{\mathcal{F}_1/r} e^{\frac{\alpha_r}{r}} \mathcal{O}^h_\mu (u)^{-1} e^{-\frac{\alpha_r}{r}} u^{-\mathcal{F}_1/r}
\end{equation}
The conjugation of $\mathcal{O}$ by the operator $e^{\alpha_r/r}$ is given by formula \eqref{eq:conj2}. The conjugation with \( u^{\mathcal{F}_1/r} \) annihilates the factor $u^t$ and produces a factor $u^{-\mu/r}$, which simplifies with $u^{\mu/r}$. This proves equation \eqref{eq:A-inv}. Equation \eqref{eq:A-invsigma} is proved in the same way starting from
\begin{equation}
\mathcal{A}^{\sigma}_{\langle \mu \rangle}(u,\mu )^{-1} = u^{\mu/r} \mu\binom{\mu - 1}{[\mu]} u^{\mathcal{F}_1/r} e^{\frac{\alpha_r}{r}} {\mathcal{O}^{\sigma}_\mu (u)}^{-1} e^{-\frac{\alpha_r}{r}} u^{-\mathcal{F}_1/r}
\end{equation}
and using the conjugation given by formula \eqref{eq:conj4}. The proposition is proved.
\end{proof}

\section{Quasi-polynomiality results}\label{sec:Poly}

In this section we state and prove the quasi-polynomiality property for monotone and strictly monotone orbifold Hurwitz numbers.
\begin{definition}
We define the connected operators \( \left\langle \prod_{i=1}^n \mathcal{A}_{\eta_i}(u, \mu_i) \right\rangle^\circ \) in terms of the disconnected correlator \( \left\langle \prod_{i=1}^n \mathcal{A}_{\eta_i}(u, \mu_i) \right\rangle^\bullet \) by means of the inclusion-exclusion formula, see, e.~g., \cite{DKOSS,DLPS}.
\end{definition}
The monotone Hurwitz numbers are expressed in terms of connected correlators as
\begin{align}
h^{\circ, r, \leq}_{g;\vec\mu} &= [u^{2g-2+l(\vec\mu)}].\prod_{i=1}^{l(\vec\mu )} \binom{\mu_i  + [\mu_i ]  }{\mu_i} \corc{ \prod_{i=1}^{l(\vec\mu)} \mathcal{A}^h_{\langle \mu_i \rangle }(u, \mu_i) }\\
h^{\circ, r, <}_{g;\vec\mu} &=  [u^{2g-2+l(\vec\mu)}]. \prod_{i=1}^{l(\vec\mu)} 
\binom{\mu_i - 1}{[\mu_i ]}
\corc{ \prod_{i=1}^{l(\vec\mu)} \mathcal{A}^{\sigma}_{\langle \mu_i \rangle }(u, \mu_i) }
\end{align}
We are now ready to state and prove the main result of the paper.
\begin{theorem}[Quasi-polynomiality for monotone and strictly monotone orbifold Hurwitz numbers] \label{Poly} For \( 2g - 2 + l(\vec\mu ) \geq0 \), the monotone and strictly monotone orbifold Hurwitz numbers can be expressed as follows:
\begin{align}
h^{\circ, r, \leq}_{g;\vec\mu} &= \prod_{i=1}^{l(\vec\mu )} \binom{\mu_i  + [\mu_i ]  }{\mu_i} P^{\langle \vec\mu \rangle}_{\leq}(\mu_1, \dots, \mu_{l(\vec\mu)}) \label{eq:Polyh} \\
h^{\circ, r, <}_{g;\vec\mu} &=   \prod_{i=1}^{l(\vec\mu )} 
\binom{\mu_i - 1}{[\mu_i ]}
P^{\langle \vec\mu\rangle}_{<}(\mu_1, \dots, \mu_{l(\vec\mu)}) \label{eq:Polysigma}
\end{align}
where $P^{\langle \vec\mu\rangle}_{<}$ and $P^{\langle \vec\mu\rangle}_{\leq}$ are polynomials of degree $3g-3+l(\vec\mu)$ depending on the parameters $\langle \mu_1 \rangle, \dots \langle \mu_{l(\vec\mu)} \rangle $ and \( \mu = r[\mu ] + \langle \mu \rangle \) denotes the euclidean division by $r$.
\end{theorem}
\begin{remark}
The two statements of theorem \ref{Poly} confirm respectively conjecture 23 in \cite{DoKarev} and conjecture 12 in \cite{DoManescu}. 
Note that the small difference in the conjecture 23 does not affect quasi-polynomiality since the polynomials $P_{\leq}$ depend on the parameters $\langle \mu \rangle$. Conjecture 12 is stated for Grothendieck dessin d'enfants, which indeed correspond to strictly monotone Hurwitz numbers by the Jucys correspondence (see for example \cite{ALS} for details).
\end{remark}
\begin{remark}\label{rem:mu-floormu-polynomiality}
	Note that since we allow the coefficients of the polynomials  $P^{\langle \vec\mu \rangle}_{\leq}$ and $P^{\langle \vec\mu \rangle}_{<}$ to depend on $\langle \vec\mu \rangle$, we can equivalently consider them as polynomials in $[\mu_1],\dots,[\mu_n]$, $n:=l(\vec\mu)$. The latter way is more convenient in the proof. 
\end{remark}

\begin{proof} We will show that,
for fixed \( \eta_i \), the connected correlator $ \left\langle \prod_{i=1}^n \mathcal{A}_{\eta_i}(u, \mu_i) \right\rangle^{\circ}$ is a power series in \( u\) with polynomial coefficients in all \( \mu_i \), for both the operators $\mathcal{A}^h$ and $\mathcal{A}^{\sigma}$.
As these are symmetric functions in the \( \mu_i \), it is sufficient to prove polynomiality in \( \mu_1\). Indeed, if a symmetric function $P(\mu_1, \dots, \mu_n)$ is polynomial in the first variable, it can be written in the form $P(\mu_1, \dots, \mu_n) = \sum_{k=0}^d  a_k(\mu_2, \dots, \mu_n) \mu_1^k$. To check that each coefficient of $P$ is also polynomial in $\mu_2$, we can compute the values of $P$ at the points $\mu_1=1,\dots,d+1$ and show that these values are polynomial in $\mu_2$. But the values of $P$ at these particular values of $\mu_1$ can be computed using the symmetry of $P$ as $P(\mu_2,\dots,\mu_n,\mu_1)$, so they are polynomial in $\mu_2$. Proceeding this way, we establish polynomiality of $P$ in all arguments.

 We will first consider the disconnected correlator $ \left\langle \prod_{i=1}^n \mathcal{A}_{\eta_i}(u, \mu_i) \right\rangle^{\bullet} $ where, setting \( \mu_i = \nu_i r + \eta_i\) to stress the independence the parameters $\nu_i=[\mu_i]$ and $\eta_i=\langle \mu_i \rangle$ here, the operator $\mathcal{A}$ is either
\begin{equation}
\mathcal{A}^h_{\eta_i}(u,\mu_i) =  \sum_{t_i \in \Z}^{\infty} \sum_{v_i=t_i}^{\infty} \frac{ (\nu_i  \! + \! \mu_i \! + \! 1)_{v_i - 1}}{(\nu_i + 1)_{t_i} } [z^{v_i-t_i}] \mathcal{S}(uz)^{\mu_i - 1} \mathcal{S}(ruz)^{t_i+\nu_i} \mathcal{E}_{t_ir - \eta_i}(uz)
\end{equation}
in the monotone case or
\begin{equation}
\mathcal{A}^{\sigma}_{\eta_i}(u,\mu_i) =  \sum_{t_i = - \infty}^{\mu_i} \sum_{v_i=t_i}^{\mu_i} \frac{ ( \mu_i \! - \! \nu_i - \! (v_i -1))_{v_i-1}}{(\nu_i + 1)_{t_i} } [z^{v_i-t_i}] \mathcal{S}(uz)^{-\mu_i - 1} \mathcal{S}(r uz)^{t+\nu_i} \mathcal{E}_{tr - \eta_i}(uz)
\end{equation}
in the strictly monotone case. 
In both cases, if we expand the product of all the \( t\)-sums in the disconnected correlator, we get the condition \( \sum_{i=1}^{l(\mu )} (t_i r - \eta_i) = 0\), as the total energy of the operators in
a given monomial must be zero. Furthermore, \( t_1r - \eta_1 \geq 0 \), since the first \( \mathcal{E} \) would get annihilated by the covacuum otherwise, 
and \( t_i \geq -\nu_i \) (otherwise the symbol $1/(\nu_i + 1)_{t_i}$ vanishes), so if we fix \( \eta_1, \nu_2, \eta_2, \dotsc, \nu_n, \eta_n \), the \( t_1 \)-sum becomes finite. Since the power of \( u\) is fixed, it also gives a bound on the degree in $\nu_1$. So the coefficient of a particular power of $u$ in the disconnected correlator $ \left\langle \prod_{i=1}^n \mathcal{A}_{\eta_i}(u, \mu_i) \right\rangle^{\bullet} $ is a rational function in $\nu_1$.

Because the coefficients are rational functions, we can extend them to the complex plane, and it makes sense to talk about their poles. The only possible poles must come from \( \frac{1}{(\nu +1)_t} \) (because we only look at non-negative exponents of \( u\)), and all of these poles are simple. Let us calculate the residue at \( \nu = - l\), for \( l = 1, 2, \dotsc \)
\begin{lemma}\label{lem:A-res}
The residue of the \( \mathcal{A} \)-operators is, up to a linear multiplicative constant, equal to the inverse of the operator with a negative argument. More precisely,
\begin{align}
\Res_{\nu = - l} \mathcal{A}^h_\eta (u, \nu r + \eta ) &= \frac{1}{l r - \eta} \mathcal{A}^h_{-\eta}(u,lr - \eta )^{-1} &&\text{if } \eta \neq 0 \label{eq:res1}\\
\Res_{\nu = - l} \mathcal{A}^h_0 (u, \nu r) &= \frac{1}{lr(r + 1)} \mathcal{A}^h_0 (u,lr)^{-1} && \text{if } \eta = 0 \label{eq:res2}\\
\Res_{\nu = - l} \mathcal{A}^{\sigma}_\eta (u, \nu r + \eta ) &= \frac{1}{lr - \eta} \mathcal{A}^{\sigma}_{-\eta}(u,lr - \eta )^{-1} &&\text{if } \eta \neq 0 \label{eq:res3}\\
\Res_{\nu = - l} \mathcal{A}^{\sigma}_0 (u, \nu r) &= \frac{1}{lr(r-1)} \mathcal{A}^{\sigma}_0 (u,lr)^{-1} && \text{if } \eta = 0 \label{eq:res4}
\end{align}
\end{lemma}
\begin{proof}
Let us prove equations \eqref{eq:res1} and \eqref{eq:res2} together. The only contributing terms have \( t \geq l\), so we calculate
\begin{align}
\Res_{\nu = -l} &\mathcal{A}^h_\eta (u, \mu )\\
&= \sum_{t = l}^{\infty} \sum_{v=t}^{\infty} \frac{ (\nu  \+ \mu \! + \! 1)_{v - 1} (\nu \+ l)}{(\nu \+ 1)_t } [x^{v-t}] \mathcal{S}(xu)^{\mu - 1} \mathcal{S}(rx u)^{t+\nu} \mathcal{E}_{tr - \eta}(x u) \bigg|_{\nu = -l}\\
&= \sum_{t = l}^{\infty} \sum_{v=t}^{\infty} \frac{(\mu \mi l \+ 1)_{v - 1}}{(1 \mi l)_{l - 1} (t \mi l)!} (-1)^{v-t}[x^{v-t}] \mathcal{S}(-xu)^{\mu - 1} \mathcal{S}(-rxu)^{t-l} \mathcal{E}_{tr - \eta}(-xu)\\
&= \sum_{t = 0}^{\infty} \sum_{v=t}^{\infty} \frac{(-1)^{l+v-t-1} (\mu \mi l \+ 1)_{v + l - 1}}{(l \mi 1)! t!}[x^{v-t}] \mathcal{S}(xu)^{\mu - 1} \mathcal{S}(rxu)^t \mathcal{E}_{tr -\mu}(-xu)
\end{align}
where we kept writing \( \mu \) for \( -lr + \eta \). As this is negative, however, it makes sense to rename it \( \mu = -\lambda \). Substituting and collecting the minus signs from the Pochhammer symbol, we get
\begin{align}
\Res_{\nu = -l} &\mathcal{A}^h_\eta (u, \mu )\\
&= \sum_{t = 0}^\lambda \sum_{v=t}^\lambda \frac{(-1)^t (\lambda \+ 1 \mi v)_{v + l - 1}}{(l \mi 1)! t!}[x^{v-t}] \mathcal{S}(ux)^{- \lambda - 1} \mathcal{S}(rux)^t \mathcal{E}_{tr +\lambda}(-ux)\\
&= \sum_{t = 0}^\lambda \sum_{v=t}^\lambda \frac{(-1)^t (\lambda \+ l \mi 1)!}{(l \mi 1 )! t! (\lambda \mi v)!}[x^{v-t}] \mathcal{S}(ux)^{- \lambda - 1} \mathcal{S}(rux)^t \mathcal{E}_{tr +\lambda}(-ux)
\end{align}
Because \( \lambda = lr - \eta \), we have \( l = [\lambda ] + 1-\delta_{\eta 0} \) and \( \eta = - \langle \lambda \rangle \). Recalling equation \eqref{eq:A-inv}, we obtain the result. Equations \eqref{eq:res3} and \eqref{eq:res4} follow from the analogous computation of the residue and the comparison with equation \eqref{eq:A-invsigma}.
\end{proof}
In the following we will use the notation $\mathcal{A}$ and $\mathcal{D}$ without specifying the symmetric polynomial chosen, since the argument is valid for both the choices of $(\mathcal{A}^h, \mathcal{D}^h)$ and $(\mathcal{A}^\sigma, \mathcal{D}^{\sigma})$. Lemma \ref{lem:A-res} implies that we can express the residues in \( \mu_1 \) 
of the disconnected correlator as follows:
\begin{equation}
\Res_{\nu_1 = - l}\cord{ \prod_{i=1}^n \mathcal{A}_{\eta_i}(u, \mu_i)} = 
c(l,\eta_1 )
\cord{\mathcal{A}_{-\eta_1}(u,lr-\eta_1 )^{-1} \prod_{i=2}^n \mathcal{A}_{\eta_i}(u,\mu_i)}.
\end{equation}
where \( c(l,\eta_1 )\) is the coefficient in lemma~\ref{lem:A-res}.
 Recalling equations~\eqref{eq:GenSerFock} and \eqref{eq:FormulaForGeneratingFunction} for the monotone case and equations~\eqref{eq:GenSerFocksigma} and \eqref{eq:FormulaForGeneratingFunctionsigma} for the strictly monotone case and realising that the inverse \( \mathcal{A}\)-operator is given by the same conjugations as the normal \( \mathcal{A}\)-operator, but starting from \( \alpha_\mu \) instead of \( \alpha_{-\mu} \), we can see that this reduces to
\begin{equation}\label{eq:rescord}
\Res_{\nu_1 = - l}\cord{ \prod_{i=1}^n \mathcal{A}_{\eta_i}(u, \mu_i)} = C\cord{ e^{\frac{\alpha_{r}}{r}} \mathcal{D}(u) \alpha_{lr-\eta_1} \prod_{i=2}^n \alpha_{-\mu_i} }
\end{equation}
for some specific coefficient \( C\) that depends only on $l$ and $\eta_1$.\par
Because \( [\alpha_k, \alpha_l] = k\delta_{k+l,0} \), and \( \alpha_{lr-\eta_1} \) annihilates the vacuum, this residue is zero unless one of the \( \mu_i \) equals \( lr - \eta_1 \) for \( i \geq 2\).\par
Now return to the connected correlator. It can be calculated from the disconnected one by the inclusion-exclusion principle, so in particular it is a finite sum of products of disconnected correlators. Hence the connected correlator is also a rational function in $\nu_1$, and all possible poles must be inherited from the disconnected correlators. So let us assume \( \mu_i = lr-\eta_1 \) for some \( i \geq 2\). Then we get a contribution from \eqref{eq:rescord}, but this is canceled exactly by the term coming from
\begin{align}
\Res_{\nu_1 = - l}&\cord{\mathcal{A}_{\eta_1}(u, \mu_1)\mathcal{A}_{-\eta_1}(u,lr-\eta_1) }\cord{\prod_{\substack{2 \leq j \leq n \\ j \neq i}} \mathcal{A}_{\eta_j}(u, \mu_j)} \\
&= C\cord{ e^{\frac{\alpha_{r}}{r}} \mathcal{D}(u) \alpha_{lr-\eta_1} \alpha_{-(lr-\eta_1)} }\cord{ e^{\frac{\alpha_{r}}{r}} \mathcal{D}(u) \alpha_{lr-\eta_1} \prod_{\substack{2 \leq j \leq n \\ j \neq i}} \alpha_{-\mu_j} }
\end{align}
Hence, the connected correlator has no residues, which proves it is polynomial in \( \nu_1 \). Therefore, it is also a polynomial in $\mu_1$, see remark~\ref{rem:mu-floormu-polynomiality}. This completes the proof of the polynomiality. 

Now, once we know that the coefficient of $u^{2g-2+n}$, $2g-2+n\geq 0$, of a connected correlator $ \left\langle \prod_{i=1}^n \mathcal{A}_{\eta_i}(u, \mu_i) \right\rangle^{\circ} $ is a polynomial in $\mu_1,\dots,\mu_n$, or, equivalently, in $\nu_1,\dots,\nu_n$, we can compute its degree. The argument is the same in both cases, monotone and strictly monotone, so let us use the formulas for the $\mathcal{A}^h$-operators. We can compute the degree of the connected correlator considered as a rational function. Once we know that it is a polynomial, we obtain the degree of the polynomial.  For the computation of the degree in $\nu_1,\dots,\nu_n$ it is sufficient to observe that $\sum_{i=1}^n (v_i-t_i) = 2g-2+n$, therefore  $\prod_{i=1}^n  (\nu_i  \! + \! \mu_i \! + \! 1)_{v_i - 1}/(\nu_i + 1)_{t_i}  $ has degree $2g-2$. Moreover, the leading term in $ \left\langle \prod_{i=1}^n \mathcal{E}_{t_ir-\eta_i}(uz)\right\rangle^{\circ} $ has degree $n-2$ in $uz$ and $n-1$ in $\nu_1,\dots,\nu_n$, and the coefficient of $(uz)^{2g}$ in  the product of $\mathcal{S}\cdot \prod_{i=1}^n\mathcal{S}(uz)^{\mu_i - 1} \mathcal{S}(ruz)^{t_i+\nu_i}$, where $\mathcal{S}$ without an argument denotes the $\mathcal{S}$-functions coming from the connected correlator $ \left\langle \prod_{i=1}^n \mathcal{E}_{t_ir-\eta_i}(uz)\right\rangle^{\circ} $ divided by its leading term, is a polynomial of degree $2g/2=g $ in $\nu_1,\dots,\nu_n$. So, the total degree in $\nu_1,\dots,\nu_n$ is equal to $2g-2+n-1+g=3g-3+n$. 

This completes the proof of the theorem.
\end{proof}

\subsection{Quasi-polynomiality for the usual orbifold Hurwitz numbers}

In the case of the usual orbifold Hurwitz numbers, quasi-polynomiality was already known, see \cite{BHLM,DLN,DLPS}. However, all known proofs use either the Johnson-Pandharipande-Tseng formula ~\cite{JPT} (the ELSV formula~\cite{ELSV} for $r=1$) or very subtle analytic tools due to Johnson~\cite{Joh2} (Okounkov-Pandharipande~\cite{OP} for $r=1$). In the second approach, presented in~\cite{DKOSS,DLPS}, the analytic continuation to the integral points outside the area of convergence requires an extra discussion, which is so far omitted. So, it would be good to have a more direct combinatorial proof of quasi-polynomiality for usual orbifold Hurwitz numbers, and we will reprove it here using the same technique as for the (strictly) monotone orbifold Hurwitz numbers. 

\begin{definition}
The usual orbifold \( \mathcal{A}\)-operators are given by
\begin{equation}
\mathcal{A}_{\langle \mu \rangle}(u,\mu ) \coloneqq r^{-\frac{\langle \mu \rangle}{r}} \mathcal{S}(ru\mu )^{[\mu ]} \sum_{t \in \Z} \frac{\mathcal{S} (ru\mu )^t \mu^{t-1}}{([\mu ]+1)_t} \mathcal{E}_{tr-\langle \mu \rangle}(u\mu )
\end{equation}
\end{definition}
\begin{remark}
Up to slightly different notation and a shift by one in the exponent of \( \mu \), these are the \( \mathcal{A}\)-operators of \cite{DLPS}.
\end{remark}
The importance of these operators is given in the following proposition:
\begin{proposition}\cite[proposition 3.1]{DLPS}
The generating function for disconnected orbifold Hurwitz can be expressed in terms of the \( \mathcal{A}\)-operators by:
\begin{equation}\label{eq:GenSerFockusual}
H^\bullet(u,\vec\mu ) = \sum_{g=0}^\infty h^{\circ}_{g;\vec\mu} u^b = r^{\sum_{i=1}^{l (\vec\mu)} \frac{\langle \mu_i\rangle}{r}} \prod_{i=1}^{l(\vec\mu )} \frac{u^{\frac{\mu_i}{r}}\mu_i^{[\mu_i ]}}{[\mu_i ]!} \cord{ \prod_{i=1}^{l(\vec\mu )} \mathcal{A}_{\langle \mu_i \rangle} (u,\mu_i )}
\end{equation}
\end{proposition}
The proof of this proposition amounts to the calculation
\begin{equation}\label{eq:usualAconj}
r^{\frac{\langle \mu \rangle}{r}} \frac{u^{\frac{\mu}{r}}\mu^{[\mu ]}}{[\mu ]!} \mathcal{A}_{\langle \mu \rangle}(u,\mu ) = u^{\frac{\mathcal{F}_1}{r}} e^{\frac{\alpha_r}{r}} e^{u\mathcal{F}_2} \alpha_{-\mu} e^{-u\mathcal{F}_2} e^{-\frac{\alpha_r}{r}} u^{-\frac{\mathcal{F}_1}{r}}
\end{equation}
With these data, we can start our scheme of proof.
\begin{lemma}
The inverse of \( \mathcal{A}_{\langle \mu \rangle} (u,\mu )\) (in the same sense as before) is given by
\begin{equation}
\mathcal{A}_{\langle \mu \rangle} (u,\mu )^{-1} = \frac{r^{\frac{\langle \mu \rangle}{r}}}{[\mu ]!} \sum_{t \geq 0} (-1)^t \frac{\mathcal{S} (ru\mu )^t \mu^{t+[\mu ]}}{t!} \mathcal{E}_{tr+ \mu}(-u\mu )
\end{equation}
\end{lemma}
\begin{proof}
The proof is very analogous to the proof of \cite[proposition 3.1]{DLPS}.\par
We do the same commutation as for the \( \mathcal{A}\)-operators, but starting from \( \alpha_\mu \). First recall \cite[equation (2.14)]{OP}:
\begin{equation}
e^{u\mathcal{F}_2} \alpha_{\mu} e^{-u\mathcal{F}_2} = \mathcal{E}_\mu (-u\mu )
\end{equation}
The second conjugation gives
\begin{align}
e^{\frac{\alpha_r}{r}} e^{u\mathcal{F}_2} \alpha_{\mu} e^{-u\mathcal{F}_2} e^{-\frac{\alpha_r}{r}} &= e^{\frac{\alpha_r}{r}} \mathcal{E}_\mu (-u\mu ) e^{-\frac{\alpha_r}{r}} \\
&= \sum_{t=0}^\infty \Big( \frac{\zeta (-ru\mu )}{r}\Big)^t \frac{1}{t!} \mathcal{E}_{tr+\mu}(-u\mu )\\
&= \sum_{t=0}^\infty \frac{(-u\mu )^t \mathcal{S} (-ru\mu )^t}{t!} \mathcal{E}_{tr+\mu}(-u\mu )
\end{align}
And the third conjugation finally shifts the exponent of \( u\):
\begin{equation}
u^{\frac{\mathcal{F}_1}{r}} e^{\frac{\alpha_r}{r}} e^{u\mathcal{F}_2} \alpha_{\mu} e^{-u\mathcal{F}_2} e^{-\frac{\alpha_r}{r}} u^{-\frac{\mathcal{F}_1}{r}} = u^{-\frac{\mu}{r}}\sum_{t=0}^\infty \frac{(-\mu )^t \mathcal{S} (-ru\mu )^t}{t!} \mathcal{E}_{tr+\mu}(-u\mu )
\end{equation}
Comparing this to equation \eqref{eq:usualAconj} shows that this is the inverse of
\begin{equation}
r^{\frac{\langle \mu \rangle}{r}} \frac{u^{\frac{\mu}{r}}\mu^{[\mu ]}}{[\mu ]!} \mathcal{A}_{\langle \mu \rangle}(u,\mu )
\end{equation}
Multiplying by this coefficient finishes the proof.
\end{proof}

\begin{theorem}[Quasi-polynomiality for usual orbifold Hurwitz numbers] \label{Polyusual} For \( 2g - 2 + l(\mu ) \geq0 \), the usual orbifold Hurwitz numbers can be expressed as follows:
\begin{align}
h^{\circ,r}_{g;\vec\mu} &= r^{\sum_{i-1}^{l (\vec\mu)} \frac{\langle \mu_i\rangle}{r}} \prod_{i=1}^{l(\vec\mu )} \frac{u^{\frac{\mu_i}{r}}\mu_i^{[\mu_i ]}}{[\mu_i ]!}  P^{\langle \vec\mu \rangle}(\mu_1, \dots, \mu_{l(\vec\mu)}) \label{eq:Polyusual}
\end{align}
where $P^{\langle \mu\rangle}$ are polynomials of degree $3g-3+l(\vec\mu)$ whose coefficients depend on the parameters $\langle \mu_1 \rangle, \dots \langle \mu_{l(\mu)} \rangle $ and \( \mu = r[\mu ] + \langle \mu \rangle \) denotes the euclidean division by $r$.
\end{theorem}

\begin{remark}
As stated before, this result is not new. It has been proved in several ways in \cite{BHLM,DLN,DLPS}. We add this new proof for completeness.
\end{remark}

\begin{proof} We will show that, for fixed \( \eta_i \), the connected correlator $ \left\langle \prod_{i=1}^n \mathcal{A}_{\eta_i}(u, \mu_i) \right\rangle^{\circ}$, $n=l(\vec\mu)$, is a power series in \( u\) with polynomial coefficients in all \( \mu_i \) for the operators $\mathcal{A}$.
As these are symmetric functions in the \( \mu_i \), it is again sufficient to prove polynomiality in \( \mu_1\), or, equivalently (see remark~\ref{rem:mu-floormu-polynomiality}) in $\nu_1:=[\mu_1]$.\par

 We will first consider the disconnected correlator $ \left\langle \prod_{i=1}^n \mathcal{A}_{\eta_i}(u, \mu_i) \right\rangle^{\bullet} $ where, setting \( \mu_i = \nu_i r + \eta_i\), the operator $\mathcal{A}$ is 
\begin{equation}
\mathcal{A}_{\eta_i}(u,\mu_i ) \coloneqq r^{-\frac{\eta_i}{r}} \mathcal{S}(ru\mu_i )^{\nu_i} \sum_{t_i \in \Z} \frac{\mathcal{S} (ru\mu_i )^{t_i} \mu_i^{t_i-1}}{(\nu_i+1)_{t_i}} \mathcal{E}_{t_ir-\eta_i}(u\mu_i )
\end{equation}
If we expand all of the \( t\)-sums in the disconnected correlator, we get the condition \( \sum_{i=1}^{l(\mu )} (t_i r - \eta_i) = 0\), as the total energy of the operators in a given monomial must be zero. Furthermore, \( t_1r - \eta_1 \geq 0 \), since the first \( \mathcal{E} \) would get annihilated by the covacuum otherwise, and \( t_i \geq -\nu_i \) (otherwise the symbol $1/(\nu_i+1)_{t_1}$ vanishes), so if we fix \( \eta_1, \nu_2, \eta_2, \dotsc, \nu_n, \eta_n \), the \( t_1 \)-sum becomes finite. Since the power of \( u\) is fixed, it also gives a bound on the degree in $\nu_1$. So the coefficient of a particular power of $u$ in the disconnected correlator $ \left\langle \prod_{i=1}^n \mathcal{A}_{\eta_i}(u, \mu_i) \right\rangle^{\bullet} $ is a rational function in $\nu_1$. \par

Again, because the coefficients are rational functions, we can extend them to the complex plane, and it makes sense to talk about poles. The only possible poles must come from \( \frac{1}{(\nu +1)_t} \) or \( \mu = 0\). These poles are all simple, except possibly for the last case. Let us calculate the residue at \( \nu = - l\), for \( l = 1, 2, \dotsc \).
\begin{lemma}\label{lem:usualA-res}
The residue of the \( \mathcal{A} \)-operators at negative integers is, up to a multiplicative constant, equal to the inverse of the operator with a negative argument. More precisely,
\begin{align}
\Res_{\nu = - l} \mathcal{A}_\eta (u, \nu r + \eta ) &= \mathcal{A}_{-\eta}(u,lr - \eta )^{-1} &&\text{if } \eta \neq 0 
\\
\Res_{\nu = - l} \mathcal{A}_0 (u, \nu r) &= \frac{1}{r} \mathcal{A}_0 (u,lr)^{-1} && \text{if } \eta = 0 
\end{align}
\end{lemma}
\begin{proof}
Let us prove both equations together. The only contributing terms have \( t \geq l\), so we calculate
\begin{align}
\Res_{\nu = -l} \mathcal{A}_\eta (u, \mu )& = r^{-\frac{\eta}{r}} \mathcal{S}(ru\mu )^{\nu} \sum_{t \geq l} \frac{\mathcal{S} (ru\mu )^t \mu^{t-1} (\nu + l)}{(\nu+1)_t} \mathcal{E}_{tr-\eta}(u\mu ) \bigg|_{\nu = -l}\\
& = r^{-\frac{\eta}{r}} \mathcal{S}(ru\mu )^{-l} \sum_{t \geq l} \frac{\mathcal{S} (ru\mu )^t \mu^{t-1} }{(1-l)_{l-1} (t-l)!} \mathcal{E}_{tr-\eta}(u\mu )
\end{align}
where we kept writing \( \mu \) for \( -lr + \eta \). As this is negative, however, it makes sense to rename it \( \mu = -\lambda \). Substituting and collecting the minus signs from the Pochhammer symbol, we get
\begin{align}
\Res_{\nu = -l} \mathcal{A}_\eta (u, \mu )&= \frac{(-1)^{l-1}r^{-\frac{\eta}{r}}}{(l-1)!} \mathcal{S}(ru\lambda )^{-l} \sum_{t \geq l} (-1)^{t-1} \frac{\mathcal{S} (ru\lambda )^t \lambda^{t-1} }{(t-l)!} \mathcal{E}_{tr-\eta}(-u\lambda )\\
&=\frac{r^{-\frac{\eta}{r}}}{(l-1)!} \sum_{t \geq 0} (-1)^t \frac{\mathcal{S} (ru\lambda )^t \lambda^{t+l-1} }{(t-l)!} \mathcal{E}_{tr+ \lambda}(-u\lambda )
\end{align}
Because \( \lambda = lr - \eta \), we have \( l = [\lambda ] + 1-\delta_{\eta 0} \) and \( \eta = - \langle \lambda \rangle \). Recalling equation \eqref{eq:A-inv}, we obtain the result.
\end{proof}
Because of lemma \ref{lem:usualA-res}, we can express the residues in \( \mu_1 \) of the disconnected correlator as follows:
\begin{equation}
\Res_{\nu_1 = - l}\cord{ \prod_{i=1}^n \mathcal{A}_{\eta_i}(u, \mu_i)} = c(\eta_1 )\cord{\mathcal{A}_{-\eta_1}(u,lr-\eta_1 )^{-1} \prod_{i=2}^n \mathcal{A}_{\eta_i}(u,\mu_i)}
\end{equation}
where \( c(\eta_1 )\) is the coefficient in lemma~\ref{lem:usualA-res}. Recalling equation~\eqref{eq:GenSerFockusual} and realising that the inverse \( \mathcal{A}\)-operator is given by the same conjugations as the normal \( \mathcal{A}\)-operator, but starting from \( \alpha_\mu \) in stead of \( \alpha_{-\mu} \), we can see that this reduces to
\begin{equation}\label{eq:rescordusual}
\Res_{\nu_1 = - l}\cord{ \prod_{i=1}^n \mathcal{A}_{\eta_i}(u, \mu_i)} = C\cord{ e^{\frac{\alpha_{r}}{r}} \mathcal{D}(u) \alpha_{lr-\eta_1} \prod_{i=2}^n \alpha_{-\mu_i} }
\end{equation}
for some specific coefficient \( C\) that depends only on $\eta_1$ and $l$.\par
For the pole at zero, we see the only contributing terms must have \( t \leq 0\), but we also need \( tr - \eta \geq 0\), in order for the \( \mathcal{E} \) not to get annihilated by the covacuum. Therefore, we need only consider the case \( \eta = 0\) and the term \( t =0\). However, this term in \( \Big\langle \prod_{i=1}^n \mathcal{A}_{\eta_i}(u, \mu_i)\Big\rangle^\bullet \) cancels against the term coming from
\begin{equation}
\cord{ \mathcal{A}_{\eta_1}(u, \mu_1)} \cord{ \prod_{i=2}^n \mathcal{A}_{\eta_i}(u, \mu_i)}
\end{equation}
as that has exactly the same conditions \( \eta = t = 0\) in order for the first correlator not to vanish.\par
The rest of the proof is completely parallel to that of theorem~\ref{Poly}, only the computation of the degree of the polynomial makes some difference.

The degree of the coefficient of $u^{2g-2+n}$, $2g-2+n\geq 0$, of a connected correlator $ \left\langle \prod_{i=1}^n \mathcal{A}_{\eta_i}(u, \mu_i) \right\rangle^{\circ} $ can be computed in the following way. The coefficient $\prod_{i=1}^n \mu_i^{t_i-1}/(\nu_i+1)_{t_i}$ has degree $-n$ in $\nu_1,\dots,\nu_n$ and degree $0$ in $u$. The leading term of the connected correlator $ \left\langle \prod_{i=1}^n \mathcal{E}_{t_ir-\eta_i}(u\mu_i)\right\rangle^{\circ} $ has degree $n-1+n-2=2n-3$ in $\nu_1,\dots,\nu_n$ and degree $n-2$ in $u$. 
The coefficient of $u^{2g}$ in the series $\mathcal{S}\cdot \prod_{i=1}^n \mathcal{S}(ru\mu_i )^{\nu_i+t_i}$, where $\mathcal{S}$ without argument denotes the $S$-functions coming from the connected correlator $ \left\langle \prod_{i=1}^n \mathcal{E}_{t_ir-\eta_i}(u\mu_i)\right\rangle^{\circ} $ divided by its leading term, is a polynomial of degree $(3/2)\cdot 2g=3g $ in $\nu_1,\dots,\nu_n$. So, the total degree in $\nu_1,\dots,\nu_n$ is equal to $-n+2n-3+3g=3g-3+n$. 

This completes the proof of the theorem.
\end{proof}

\section{Correlation functions on spectral curves}\label{sec:Correlation} 

In this section we explain the relation of the polynomiality statements with the fact that the $n$-point generation functions can be represented via correlation functions defined on the $n$-th cartesian power of a spectral curve. The results concerning the monotone and strictly monotone Hurwitz numbers in this section are new, while in the case of usual Hurwitz numbers it is well-known and we recall it here for completeness. 

The set-up for the problems considered in this paper is the following: We consider a spectral curve $\CP1$ with a global coordinate $z$, with a function $x=x(z)$ on it. Let $\{p_0,\dots,p_{r-1}\}$ be the set of the $z$-coordinates of the critical points of $x$.  We consider the $n$-point generating function of a particular Hurwitz problem, for a fixed genus $g$, and we want it to be an expansion of a symmetric function on $\left(\CP1\right)^{\times n}$ of a particular type:
\begin{equation}\label{eq:RepresentationOnSpectralCurve}
\sum_{\substack{0\leq \alpha_1,\dots,\alpha_n\leq r-1}} P_{\vec\alpha} \left(\frac{d}{dx_1},\dots,\frac{d}{dx_n}\right) \prod_{i=1}^n \xi_{\alpha_i}(x_i)
\end{equation}
Here the $P_{\vec\alpha}$ are polynomials in $n$ variables of degree $3g-3+n$, and the functions $\xi_{\alpha}(x)$ are defined as (the expansions of) some functions that form a convenient basis in the space spanned by $1/(p_\alpha - z )$, $\alpha=0,\dots,r-1$.\par
The reason we are interested in the particular degree \( 3g-3+n\), is in short due to this being the dimension of the moduli space of curves \( \overline{\mathcal{M}}_{g,n}\). Somewhat more explicitly, we expect an ELSV-type formula to hold, as it does in the usual orbifold case\textemdash{}the ELSV-formula itself for \( r = 1\) \cite{ELSV} and the Johnson-Pandharipande-Tseng formula for general \( r\) \cite{JPT}, for more explanations and examples we refer to~\cite{Ey,DOSS,LPSZ,DLPS,ALS}. The topological recursion implies~\cite{Ey} that the correlation differentials are given by the differentials of
\begin{equation}
\sum_{\substack{0\leq \alpha_1,\dots,\alpha_n\leq r-1}} \left[\int_{\overline{\mathcal{M}}_{g,n}} \frac{C_{\vec\alpha}}{ \prod_{i=1}^n \left(1-\psi_i  \frac{d}{dx_i}\right) } \right] \prod_{i=1}^n \xi_{\alpha_i}(x_i),
\end{equation}
where $C_{\vec\alpha}$ is some class in the cohomology of $\overline{\mathcal{M}}_{g,n}$. Because the complex cohomological degree of the \( \psi \)-classes is one, this implies that we have a  polynomial in the derivatives of degree \( \dim \overline{\mathcal{M}}_{g,n} = 3g-3+n\).

\subsection{Monotone orbifold Hurwitz numbers} In the case of the monotone orbifold Hurwitz numbers the conjectural spectral curve  is given by $x=z(1-z^r)$ \cite{DoKarev}. The conjecture on the topological recursion assumes the expansion of equation~\eqref{eq:RepresentationOnSpectralCurve} in $x_1,\dots,x_n$ near $x_1=\cdots=x_n=0$, so we have the following expected property of orbifold Hurwitz numbers:
\begin{equation}\label{eq:RepresentationOnSpectralCurveMonotone}
\sum_{\vec\mu\in\left(\mathbb{N}^\times\right)^n} h^{\circ,r,\leq}_{g;\vec{\mu}} \prod_{i=1}^n x_i^{\mu_i} = \sum_{\substack{0\leq \alpha_1,\dots,\alpha_n\leq r-1}}  P_{\vec\alpha} \left(\frac{d}{dx_1},\dots,\frac{d}{dx_n}\right) \prod_{i=1}^n \xi_{\alpha_i}(x_i).
\end{equation}

 In this case the critical points are given by $p_i = \zeta^i (r+1)^{-1/r}$, $i=0,\dots,r-1$, where $\zeta$ is a primitive $r$-th root of $1$.  This means that up to some non-zero constant factors that are not important, we have the space of functions spanned by:
\begin{equation}
\xi''_i = \frac{1}{1-\zeta^{-i}(r+1)^{1/r}z}, \qquad i=0,1,\dots,r-1
\end{equation}
Consider a non-degenerate change of basis $\xi_k' = \sum_{i=0}^{r-1}\zeta^{ki}/r \cdot \xi''_i $. We have: 
\begin{equation}
\xi_k' = \frac{\left((r+1)^{1/r}z\right)^k}{1-(r+1)z^r}, \qquad k=0,1,\dots, r-1
\end{equation}
Observe that $x=z(1-z^r)$ implies 
\begin{equation}
\frac{d}{dx} = \frac{1}{1-(r+1)z^r} \frac{d}{dz}
\end{equation}
Therefore, the functions $\xi'_k$ are given up to non-zero constant factors $C_k'$ by 
\begin{equation}
\xi'_k = C_k' \frac{d}{dx} \frac{z^{k+1}}{k+1}, \qquad k=0,1,\dots, r-1
\end{equation}
 Thus, the suitable set of basis functions for the representation of the $n$-point function in the form of equation~\eqref{eq:RepresentationOnSpectralCurveMonotone} is given by 
 \begin{equation}
 \xi_i:= \frac{d}{dx} \left(\frac{z^{i+1}}{i+1} \right), \qquad i=0,\dots,r-1
\end{equation}

\begin{lemma} \label{xih}
	For $i=0,\dots,r-1$, we have:
 \begin{equation}\label{eq:FormulaForXiMonotone}
 \xi_i(x) = \sum_{\substack{\mu =0\\r| \mu - i}}^\infty \binom{\mu + [\mu ]}{\mu } x^{\mu}
 \end{equation}	
\end{lemma}
\begin{proof} In order to compute the expansion of $z^{i+1}$ in $x$, we compute the residue:
\begin{align}
& \oint z^{i+1} \frac{dx}{x^{n+1}} = \oint \frac{1-(r+1)z^r}{(1-z^r)^{n+1} } \frac{z^{i+1}dz }{z^{n+1}} = \oint \frac{dz}{z^{n-i}} (1-(r+1)z^r) \sum_{j=0}^\infty \binom{n+j}{j} z^{rj} 
\end{align} 
This residue is nontrivial only for $n=kr+i+1$, $k\geq 0$, and in this case it is equal to the coefficient of $z^{kr}$, that is,
\begin{equation}
\binom{kr+k+i+1}{k}-(r+1)\binom{kr+k+i}{k-1} = \frac{(i+1)\cdot (kr+k+i)!}{k! (kr+i+1)!}
\end{equation}
Thus 
\begin{equation}
\frac{z^{i+1}}{i+1} = \sum_{k=0}^\infty \binom{kr+k+i}{k} \frac{x^{kr+i+1}}{kr+i+1}
\end{equation}
which implies the formula for $\xi_i= (d/dx) \left(z^{i+1}/(i+1) \right)$, $i=0,\dots,r-1$, if we set \( \mu = kr+i\). 
\end{proof}
 
The explicit formulae for the expansions of functions $\xi_i$ in the variable $x$ given by equation~\eqref{eq:FormulaForXiMonotone} imply a particular structure for the coefficients of the expansion given by equation~\eqref{eq:RepresentationOnSpectralCurve}, that is, for monotone orbifold Hurwitz numbers. In fact we have:

\begin{proposition} \label{prop:ExpansionMonotone}
	The coefficient of $x_1^{\mu_1}\cdots x_n^{\mu_n}$ of the expansion in $x_1,\dots,x_n$ near zero of an expression of the form
\begin{equation}\label{eq:ExpressionPropPolyMonotone}
 \sum_{\substack{0\leq k_1,\dots,k_n\leq r-1}} P_{k_1,\dots,k_n} \left(\frac{d}{dx_1},\dots,\frac{d}{dx_n}\right) \prod_{i=1}^n \xi_{k_i}
\end{equation}
where $P_{k_1,\dots,k_n}$ are polynomials of degree $3g-3+n$ and $\xi_k$ is equal to $\frac{d}{dx} \frac{z^{k+1}}{k+1}$, is represented as
\begin{equation}
\prod_{i=1}^n \binom{\mu_i+[\mu_i ]}{\mu_i} \cdot Q_{\langle \mu_1 \rangle,\dotsc ,\langle \mu_n \rangle} ([\mu_1],\dotsc ,[\mu_n])
\end{equation}
where $\mu_i=r[\mu_i]+\langle \mu_i \rangle$, is the euclidean division, and $Q_{\eta_1,\dots,\eta_n}$ are some polynomials of degree $3g-3+n$ whose coefficients depend on $\eta_1,\dots,\eta_n \in \{ 0, \dotsc, r-1\}$.
\end{proposition}
\begin{proof}
	The coefficient of $x^\mu$ in $(d/dx)^p \xi_q$ is non-trivial if and only if $\langle \mu \rangle +p \equiv q \mod r$. In this case, the coefficient of $x^\mu$ is equal to 
	\begin{equation}\label{eq:CoeffExpansionMonotoneHur}
	\binom{[\mu+p]+\mu+p}{[\mu+p]} (\mu+1)_p = \binom{\mu+[\mu]}{\mu} \cdot \frac{([\mu+p]+\mu+p)![\mu]!}{(\mu+[\mu])![\mu+p]!}
	\end{equation}
	Represent $p$ as $p=-\langle \mu \rangle + sr+\ell\geq 0$, $0\leq \ell\leq r-1$. Then the second factor on the right hand side of equation~\eqref{eq:CoeffExpansionMonotoneHur} can be rewritten as 
	\begin{equation}
	 \frac{({([\mu]+s)(r+1)+\ell)!}}{([\mu](r+1)+\langle \mu \rangle )!([\mu]+1)_s}
	\end{equation}
	Observe that we can cancel the factors $([\mu]+1), ([\mu]+2),\dotsc, ([\mu]+s)$ in the denominator with the factors $([\mu]+1)(r+1), ([\mu]+2)(r+1),\dots, ([\mu]+s)(r+1)$ in the numerator. Since $([\mu]+1)(r+1)>[\mu](r+1)+\langle \mu \rangle$, after this cancellation the numerator is still divisible by $([\mu](r+1)+\langle \mu \rangle )!$. So, this factor is a polynomial of degree $p$ in $[\mu]$, with the leading coefficient $(r+1)^{p+s} [\mu]^p$. 
	
	Since the only possible nontrivial coefficient of $x^\mu$ in $(d/dx)^p \xi_q$ is a common factor $\binom{\mu+[\mu]}{\mu}$ multiplied by a polynomial of degree $p$ in $[\mu]$, the coefficient of $\prod_{i=1}^n x_i^{\mu_i}$ in the whole expression~\eqref{eq:ExpressionPropPolyMonotone} is also given by a common factor  $\prod_{i=1}^n \binom{\mu_i+[\mu_i]}{\mu_i}$ multiplied by a polynomial in $[\mu_1],\dotsc,[\mu_n]$ of the same degree as $P_{k_1,\dots,k_n}$. 
\end{proof}

Thus the quasi-polynomiality property of monotone orbifold Hurwitz numbers is equivalent to the property that the $n$-point functions can be represented in a very particular way (given by equation~\eqref{eq:RepresentationOnSpectralCurveMonotone}) on the corresponding conjectural spectral curve, cf.~\cite[conjecture 23]{DoKarev}. 

\subsection{Strictly monotone orbifold Hurwitz numbers} In this case the spectral curve topological recursion follows from the two-matrix model consideration~\cite{CEO}, and it was combinatorially proved in~\cite{DOPS}, see also~\cite{DoManescu}. From these papers it does follow that the $n$-point function is represented as an expansion of the following form:
\begin{equation}\label{eq:RepresentationOnSpectralCurveStrictlyMonotone}
\sum_{\vec\mu\in\left(\mathbb{N}^\times\right)^n} h^{\circ,r,<}_{g;\vec\mu} \prod_{i=1}^n x_i^{-\mu_i} =  \sum_{\substack{0\leq \alpha_1,\dots,\alpha_n\leq r-1}}  P_{\vec\alpha} \left(\frac{d}{dx_1},\dots,\frac{d}{dx_n}\right) \prod_{i=1}^n \xi_{\alpha_i}(x_i)
\end{equation}
for the curve $x=z^{r-1}+z^{-1}$. The goal of this section is to show the equivalence of this representation to the quasi-polynomiality property of strictly monotone orbifold Hurwitz numbers.

The critical points of $x$ are given by $p_i=\zeta^i (r-1)^{-1/r}$, $i=0,\dots, r-1$, so, repeating the argument for the previous section and using that in this case
\begin{equation}
-\frac{1}{z^2} \frac{d}{dx} = \frac{1}{1-(r-1)z^r} \frac{d}{dz}
\end{equation}
we see that a good basis of functions $\xi_i$ can be chosen as 
\begin{equation}
\xi_i = \frac{1}{z^2} \frac{d}{dx} \left(\frac{z^{i+1}}{i+1}\right), \qquad i=0,\dots, r-1
\end{equation}

The expansion of these function in $x^{-1}$ near $x=\infty$ is given by the following lemma:
\begin{lemma}
	For $i=0,\dots,r-1$, we have:
	\begin{equation}\label{eq:FormulaForXiSctrictlyMonotone}
	\xi_i(x) = \sum_{\substack{\mu =1\\ r| \mu - i }}^\infty \binom{\mu -1}{[\mu ]} x^{-\mu}
	\end{equation}	
\end{lemma}

\begin{proof} We compute the coefficient of $x^{-\mu}$ as the residue
	\begin{equation}
	\oint \frac{1}{z^2} \frac{d}{dx} \left(\frac{z^{i+1}}{i+1}\right) x^{\mu-1}dx = \oint -\frac{z^{i+1}}{i+1} d \left(\frac{(1+z^r)^{\mu-1}}{z^{\mu+1}} \right)
	\end{equation}
	We see that his residue can be non-trivial only if $\mu+1 \equiv i+1 \mod r$, and in this case it is equal to $\binom{\mu-1}{[\mu]}$.
\end{proof}

The proof of the following statement repeats the proof of proposition~\ref{prop:ExpansionMonotone}.
\begin{proposition}
	 The coefficient of $x_1^{-\mu_1}\cdots x_n^{-\mu_n}$ of the expansion in $x_1^{-1},\dots,x_n^{-1}$ near infinity of an expression of the form
	 \begin{equation}\label{eq:ExpressionPropPolyStrictlyMonotone}
	 \sum_{\substack{0\leq k_1,\dots,k_n\leq r-1}} P_{k_1,\dots,k_n} \left(\frac{d}{dx_1},\dots,\frac{d}{dx_n}\right) \prod_{i=1}^n \xi_{k_i}
	 \end{equation}
	 where $P_{k_1,\dots,k_n}$ are polynomials of degree $3g-3+n$ and $\xi_k$ is equal to $\frac{1}{z^2} \frac{d}{dx} \left(\frac{z^{k+1}}{k+1}\right)$, is represented as
	 \begin{equation}
	 \prod_{i=1}^n  \binom{\mu_i-1}{[\mu_i]} \cdot Q_{\langle \mu_1 \rangle, \dotsc, \langle \mu_n \rangle} ([\mu_1],\dotsc,[\mu_n])
	 \end{equation}
	 where $\mu_i=r[\mu_i]+\langle \mu_i \rangle $ and $Q_{\eta_1,\dots,\eta_n}$ are some polynomials of degree $3g-3+n$ whose coefficients depend on $\eta_1,\dots,\eta_n \in \{ 0, \dotsc , r-1\}$.
\end{proposition}

Thus the polynomiality property of strictly monotone orbifold Hurwitz numbers is also equivalent to the property that the $n$-point functions can be represented in a very particular way (given by equation~\eqref{eq:RepresentationOnSpectralCurveStrictlyMonotone}) on the corresponding spectral curve, cf.~\cite[conjecture 12]{DoManescu}.\par
Note that \cite{DoManescu} has a binomial \( \binom{\mu_i-1}{[\mu_i-1]} \), which is equal to ours unless \( \langle \mu_i \rangle = 0\). In that case it differs by a factor \( r -1\), which can be absorbed in the polynomial.

\subsection{Usual orbifold Hurwitz numbers} The spectral curve topological recursion for the usual orbifold Hurwitz numbers is proved in~\cite{DLN,BHLM}, see also~\cite{DLPS,LPSZ}. The corresponding spectral curve is given by the formula $x=\log z - z^r$, and the computations for this curves are also performed in~\cite{SSZ} in relation to a different combinatorial problem. From these papers it does follow that the $n$-point function is represented as an expansion of the following form:
\begin{equation}\label{eq:RepresentationOnSpectralCurveUsual}
\sum_{\vec\mu\in\left(\mathbb{N}^\times\right)^n} h^{\circ,r}_{g;\vec \mu} \prod_{i=1}^n e^{\mu_ix^i} =  \sum_{\substack{0\leq \alpha_1,\dots,\alpha_n\leq r-1}}  P_{\vec\alpha} \left(\frac{d}{dx_1},\dots,\frac{d}{dx_n}\right) \prod_{i=1}^n \xi_{\alpha_i}(x_i)
\end{equation}
It also follows from these papers that the good basis of functions $\xi_i$ is given by 
\begin{equation}
\xi_i=\frac{d}{dx}\left(\frac{z^{i+1}}{i+1}\right)=\frac{z^i}{1-rz^r}, \qquad i=0,\dots,r-1
\end{equation}
 and the expansions of these functions in $e^x$ near $e^x=0$ is given by 
\begin{equation}
\xi_i(x) = \sum_{\substack{\mu = 0\\ r | \mu - i}}^{\infty} \frac{\mu^{[\mu ]}}{[\mu ]!} e^{\mu x}, \qquad i=0,\dots, r-1
\end{equation}
For these functions the differentiation with respect to $x$ is the same as the multiplication by the corresponding degree of $e^x$, so the following statement is obvious:
\begin{proposition}
	The coefficient of $e^{\mu_1x_1}\cdots e^{\mu_nx_n}$ of the expansion in $e^{x_1},\dots,e^{x_n}$ near zero of an expression of the form
	\begin{equation}\label{eq:ExpressionPropPolyUsual}
	\sum_{\substack{0\leq k_1,\dots,k_n\leq r-1}} P_{k_1,\dots,k_n} \left(\frac{d}{dx_1},\dots,\frac{d}{dx_n}\right) \prod_{i=1}^n \xi_{k_i}
	\end{equation}
	where $P_{k_1,\dots,k_n}$ are polynomials of degree $3g-3+n$ and $\xi_k$ is equal to $\frac{d}{dx} \left(\frac{z^{k+1}}{k+1}\right)$, is represented as
	\begin{equation}
	\prod_{i=1}^n  
	\frac{\mu_i^{[\mu_i]}}{[\mu_i]!}\cdot Q_{\langle \mu_i \rangle,\dotsc, \langle \mu_n \rangle} ([\mu_1],\dotsc,[\mu_n])
	\end{equation}
	where $\mu_i=r[\mu_i]+\langle \mu_i\rangle $ and $Q_{\eta_1,\dots,\eta_n}$ are some polynomials of degree $3g-3+n$ whose coefficients depend on $\eta_1,\dotsc,\eta_n \in \{ 0, \dotsc, r-1\}$.
\end{proposition}

\appendix
Thus the polynomiality property of usual orbifold Hurwitz numbers is also equivalent to the property that the $n$-point functions can be represented in a very particular way (given by equation~\eqref{eq:RepresentationOnSpectralCurveUsual}) on the corresponding spectral curve.

\section{Computations for unstable correlation function}\label{sec:Unstable}
In this section we prove that the unstable correlation differentials for the conjectural (or proved) CEO topological recursion spectral curve coincide with the expression derived from the $\mathcal{A}$-operators. These computations are performed in the case of monotone orbifold Hurwitz numbers for the cases $(g,n) = (0,1)$ and $(g,n) = (0,2)$, and for strictly monotone orbifold Hurwitz numbers for the case $(g,n) = (0,1)$.

Note that in both cases the computation of the $(0,1)$-numbers was done before, see~\cite{DoKarev,DoManescu,CEO,DOPS}. The \( (0,2)\)-calculation for the monotone Hurwitz numbers is a new result, but we learned after completing our calculation that Karev obtained the same formula independently \cite{Karev}.\par
We show these computations here to test the $\mathcal{A}$-operator formula and to demonstrate its power. The computation of the generating function for the $(0,2)$ monotone orbifold Hurwitz numbers is necessary for the conjecture on topological recursion in~\cite{DoKarev}.

\subsection{The case \texorpdfstring{$ (g,n)=(0,1)$}{(g,n)=(0,1)} } In this section we check that the spectral curve reproduces the correlation differential for $(g,n)=(0,1)$ obtained from the $\mathcal{A}$-operators of section \ref{sec:Aoperators}.

\subsubsection{The monotone case}

%
%

Since in the case of $n=1$ there is no difference between connected and disconnected Hurwitz numbers, the $(0,1)$-free energy for monotone Hurwitz numbers reads:
\begin{equation}
F_{0,1}^{\leq}(x)\coloneqq \sum_{\mu=1}^\infty [u^{-1+d/r}]H^{\bullet, r, \leq}(u,\mu) x^\mu
\end{equation}
Of course, in this formula only $\mu=[\mu]r$, $[\mu]\geq 0$, can contribute non-trivially. Let us compute what we get. We have:
\begin{align*}
& [u^{-1+d/r}]H^{\bullet, r, \leq}(u,\mu) = \frac{(\mu+[\mu])!}{\mu! [\mu]!} [u^{-1}]  \left\langle \mathcal{A}^h_{\langle \mu \rangle }(u,\mu) \right\rangle \\
& =  \frac{(\mu+[\mu])!}{\mu! [\mu]!} \cdot \frac{(\mu+[\mu]+1)_{-2}}{([\mu]+1)_0}\cdot 
[z^{-1}] \mathcal{S}(z)^{\mu-1} \mathcal{S}(rz)^{0+[\mu]} \left\langle \mathcal{E}_0(z) \right\rangle \\
& = \frac{(\mu+[\mu])!}{\mu! [\mu]!} \frac{1}{(\mu+[\mu])(\mu+[\mu]-1)}
\end{align*}
(here we used in the second line equation~\eqref{eq:FormulaForA-Operator}, where $t$ and $v$ deliberately must be equal to $0$ and $-1$ respectively).

Thus we have (replacing $\mu$ by $r[\mu]$ everywhere):
\begin{equation}
F_{0,1}^{\leq} = \sum_{[\mu]=1}^{\infty} \frac{(r[\mu]+[\mu]-2)!}{(r[\mu])! [\mu]!} x^{r[\mu]}
\end{equation}

\begin{theorem}
	We have: $\omega^{\leq}_{0,1}\coloneqq d F^{\leq}_{0,1} = -ydx$.
\end{theorem}

\begin{proof} The spectral curve gives $y=-z^r/x$. In lemma \ref{xih} we have shown that
	\begin{equation}\label{eq:Equationzi}
	z^i = \sum_{k=0}^\infty \frac{(kr+k+i-1)!}{k! (kr+i)!} ix^{kr+i}=\sum_{k=0}^\infty \frac{(kr+k+i-1)!}{(k+1)! (kr+i-1)!} \frac{(ki+i)}{(kr + i)} x^{kr+i}
	\end{equation}
So,
	\begin{align*}
	& -ydx = \sum_{j=0}^\infty \frac{(kr+k+r-1)!}{(k+1)! (kr+r-1)!} x^{kr+r-1} dx \\
	& = \sum_{k+1=1}^\infty \frac{((k+1)r+(k+1)-2)!}{(k+1)! ((k+1)r-1)!} x^{(k+1)r-1} dx = dF_{0,1}^{\leq} 
	\end{align*}
	(for the last equality we just identify $[\mu]$ with $k+1$).
\end{proof}

\subsubsection{The strictly monotone case}

Similarly, for strictly monotone Hurwitz numbers the $(0,1)$-free energy reads:
\begin{equation}
F_{0,1}^{<}(x)\coloneqq \sum_{\mu=1}^\infty [u^{-1+d/r}]H^{\bullet, r, \leq}(u,\mu) x^{-\mu} - \log(x)
\end{equation}
Again, only $\mu=[\mu]r$, $[\mu]\geq 0$ can contribute non-trivially. We have:
\begin{align*}
& [u^{-1+d/r}]H^{\bullet, r, \leq}(u,\mu) = \frac{(\mu-1)!}{(\mu - [\mu] -1)! [\mu]!} [u^{-1}]  \left\langle \mathcal{A}_{\langle \mu \rangle }(u,\mu) \right\rangle \\
& =  \frac{(\mu-1)!}{(\mu - [\mu] -1)! [\mu]!} (\mu - [\mu ] +2)_{-2} \\
& = \frac{(\mu-1)!}{(\mu - [\mu] +1)! [\mu]!} 
\end{align*}
(here we used in the second line equation~\eqref{eq:FormulaForA-Operatorsigma}, where $t$ and $v$ deliberately must be equal to $0$ and $-1$ respectively). Thus we have (replacing $\mu$ by $r[\mu]$ everywhere):
\begin{equation} \label{zeroonesigma}
dF_{0,1}^{<} = -\frac{1}{x} \sum_{[\mu]=1}^{\infty} \frac{(r[\mu])!}{([\mu]r - [\mu] +1)! [\mu]!} x^{-r[\mu]} dx - \frac{dx}{x}
\end{equation}

\begin{theorem}
	We have: $\omega^{<}_{0,1}\coloneqq d F_{0,1}^{<} = ydx$.
\end{theorem}

\begin{proof} The spectral curve reads $x = z^{r-1} + z^{-1}$ and $y=z$. Let us expand $z = \sum_{n=0}^{\infty} a_n x^n$ and compute the coefficients by 
	\begin{equation}
	a_n = \oint z \frac{dx}{x^{n+1}} = - \oint [1 - (r-1)z^r] z^n \sum_{j=0} \binom{n + j}{j} (-z^r)^j dz\\
	\end{equation}
This residue is nontrivial only for $n = -rj - 1, \, j \leq 0$, hence we should extract in the two summands the $j$-th and the $(j-1)$-st term respectively. Therefore, the residue reads
\begin{align}
&(-1)^{j-1} \left[ \binom{-rj -1 +j}{j} + (r-1)\binom{-rj - 1 + j -1}{j-1} \right]\\ =& (-1)^{j} \binom{-rj -1 +j}{j}\frac{1}{(-rj +j -1)}
\end{align}
Hence
	\begin{align*}
	 ydx = zdx &= \sum_{j=0}^\infty (-1)^{j} \binom{-rj -1 +j}{j}\frac{1}{(-rj +j -1)} x^{-jr-1} dx\\
	 &=-\frac{1}{x} \sum_{j=0}^\infty (-1)^j \frac{(-rj)_j}{j! (rj -j +1)} x^{-jr} dx \\
	& = -\frac{1}{x} \sum_{j=0}^\infty \frac{(rj)!}{j! (rj - j +1)!} x^{-jr} dx = dF^{<}_{0,1} 
	\end{align*}
where, in order to obtain the last line, we collected the minus signs from the Pochhammer symbol. For the last equality we identify $[\mu]$ with $j$ and incorporate the term $[\mu] = 0$ inside the sum in formula  \eqref{zeroonesigma}.
\end{proof}

\subsection{The case \texorpdfstring{$ (g,n)=(0,2)$}{(g,n)=(0,2)}} In this section we use equation~\eqref{eq:FormulaForGeneratingFunction} in order to check whether the holomorphic part of the expansion of the unique genus zero Bergman kernel gives the differential $d_1d_2 F^{\leq}_{0,2}$. More precisely, we prove the following theorem:

\begin{theorem}
	We have:
	\begin{equation}
	\frac{dz_1dz_2}{(z_1-z_2)^2} = \frac{dx_1dx_2}{(x_1-x_2)^2}+ d_1d_2 F^{\leq}_{0,2}(x_1,x_2)
	\end{equation}
\end{theorem}

\begin{proof} It is sufficient to prove that 
\begin{equation}\label{eq:2pointIdentityGenFunctLog}
\log(z_1-z_2) = \log(x_1-x_2) +F_{0,2}(x_1,x_2)+C_1(x_1)+C_2(x_2)
\end{equation}
where $C_1,C_2$ are some functions of one variable. 

We apply the Euler operator 
\begin{equation}
E\coloneqq x_1\frac{\partial}{\partial x_1}+x_2\frac{\partial}{\partial x_2}
\end{equation}
to both sides of this formula. Using that $\partial_x = (1-(r+1)z^r)^{-1} \partial_z$, we observe that in the coordinates $z_1,z_2$ the Euler operator has the form 
\begin{equation}
E\coloneqq\frac{1-z_1^r}{1-(r+1) z_1^r}\cdot z_1\frac{\partial}{\partial z_1}+
\frac{1-z_2^r}{1-(r+1) z_2^r}\cdot z_2\frac{\partial}{\partial z_2}
\end{equation}
We have:
\begin{align}
 E \log(z_1-z_2) &= 1 + r\cdot \frac{z_1^{r}+z_1^{r-1}z_2+\cdots + z_2^r  (r+1)z_1^rz_2^r}{(1-(r+1) z_1^r)(1-(r+1) z_2^r)} 
\\ \notag 
&= 1+ r \frac{\partial^2}{\partial x_1 \partial x_2} 
\left(\frac{z_1^{r+1}z_2}{(r+1)\cdot 1}+\frac{z_1^{r}z_2^2}{r\cdot 2}+\cdots + \frac{z_1z_2^{r+1}}{1\cdot(r+1)} - \frac{z_1^{r+1}z_2^{r+1}}{r+1}\right)\\ \notag
&=1+\frac{r}{r+1} \frac{\partial^2}{\partial x_1 \partial x_2}(z_1z_2-x_1x_2) + r \frac{\partial^2}{\partial x_1 \partial x_2} \Big( \frac{z_1^rz_2^2}{r\cdot 2}+\cdots + \frac{z_1^2z_2^r}{2\cdot r} \Big)
\end{align}
Using equation~\eqref{eq:Equationzi}, we finally obtain the following formula for $E\log(z_1-z_2)$:
\begin{equation}\label{eq:FinalFormulaEulerLHS-1}
r\cdot \sum_{\substack{i_1,i_2=1\\ i_1+i_2=r}}^{r-1}\sum_{k_1,k_2=0}^\infty 
\frac{(k_1r+k_1+i_1)!}{k_1! (k_1r+i_1)!}\frac{(k_2r+k_2+i_2)!}{k_2! (k_2r+i_2)!} 
x_1^{k_1r+i_1} x_2^{k_2r+i_2}
\end{equation}
for the degrees of $x_1,x_2$ not divisible by $r$ (Case I), and 
\begin{align}\label{eq:FinalFormulaEulerLHS-2}
\notag \frac{1}{r+1}+\frac{r}{r+1} \sum_{k_1,k_1=0}^\infty \binom{k_1r+ k_1}{k_1} \binom{k_2r+k_2}{k_2} x_1^{k_1r} x_2^{k_2r} \\
= 1 + \frac{r}{r+1} \sum_{\substack{k_1,k_1=0\\ (k_1,k_2)\neq(0,0)}}^\infty \binom{k_1r+ k_1}{k_1} \binom{k_2r+k_2}{k_2} x_1^{k_1r} x_2^{k_2r}
\end{align}
if one of the exponents, and, therefore, both of them, are divisible by $r$ (Case II).

Now we apply the Euler operator $E$ to the right hand side of equation~\eqref{eq:2pointIdentityGenFunctLog}. We obtain the following expression:
\begin{equation}
1+\tilde C_1(x_1)+\tilde C_2(x_2)+\sum_{\substack{\mu_1,\mu_2\geq 1\\r | (\mu_1+\mu_2)}} 
h^{\circ,r,\leq}_{0; (\mu_1,\mu_2)} x_1^{\mu_1} x_2^{\mu_2} (\mu_1+\mu_2)
\end{equation}
We have to prove that the sum of equations~\eqref{eq:FinalFormulaEulerLHS-1} and~\eqref{eq:FinalFormulaEulerLHS-2} is equal to this expression.

Let us compute  $h^{\circ,r,\leq}_{0; (\mu_1,\mu_2)}$. Equation~\eqref{eq:FormulaForGeneratingFunction} implies that 
\begin{equation}
h^{\circ,r,\leq}_{0; (\mu_1,\mu_2)} = \binom{\mu_1+[\mu_1]}{\mu_1} \binom{\mu_2+[\mu_2]}{\mu_2} \cdot \left\langle 
\mathcal{A}_{\langle\mu_1\rangle}(u,\mu_1) 
\mathcal{A}_{\langle\mu_2\rangle}(u,\mu_2)
\right\rangle^\circ
\end{equation}
Since we have to use connected correlators, it implies that in the $\mathcal{A}_{\langle\mu_1\rangle}$-operator we have to take only the operators $\mathcal{E}$ with the positive indices, and in the $\mathcal{A}_{\langle\mu_2\rangle}$-operator we have to take only the operators $\mathcal{E}$ with the negative indices. Specializing the formula further, and using that $[\zeta_1^0\zeta_2^0] \left\langle \mathcal{E}_{v}(\zeta_1) \mathcal{E}_{-v}(\zeta_2) \right\rangle^\circ = v$, we have:
\begin{equation}
h^{\circ,r,\leq}_{0; (\mu_1,\mu_2)} 
= \sum_{t=1}^{[\mu_2]+1} \frac{(\mu_1+[\mu_1]+t-1)!}{\mu_1!([\mu_1]+t)!} (tr-\langle \mu_1 \rangle) \frac{(\mu_2+[\mu_2]-t)!}{\mu_2!([\mu_2]+1-t)!}
\end{equation}
in Case I, and
\begin{equation}
h^{\circ,r,\leq}_{0; (\mu_1,\mu_2)} 
= \sum_{t=1}^{[\mu_2]} \frac{(\mu_1+[\mu_1]+t-1)!}{\mu_1!([\mu_1]+t)!} \cdot tr\cdot \frac{(\mu_2+[\mu_2]-t-1)!}{\mu_2!([\mu_2]-t)!}
\end{equation}
in Case II. Note that in Case II, we omit the contributions from the \( t=0\) part, as it cancels the strictly diconnected correlator in the inclusion-exclusion formula. \par
So, in order to complete the proof of the theorem we have to show that 
\begin{align} \label{eq:TechnicalEquationProof02}
 (\mu_1+\mu_2) &\sum_{t=1}^{[\mu_2]+1} \frac{(\mu_1+[\mu_1]+t-1)!}{\mu_1!([\mu_1]+t)!} (tr-\langle \mu_1 \rangle) \frac{(\mu_2+[\mu_2]-t)!}{\mu_2!([\mu_2]+1-t)!} \\ \notag
& = r\cdot \binom{\mu_1+[\mu_1]}{\mu_1} \binom{\mu_2+[\mu_2]}{\mu_2}
\end{align}
in Case I (cf. equation~\eqref{eq:FinalFormulaEulerLHS-1}) and
\begin{align}\label{eq:TechnicalEquationProof03}
 (\mu_1+\mu_2) &\sum_{t=1}^{[\mu_2]} \frac{(\mu_1+[\mu_1]+t-1)!}{\mu_1!([\mu_1]+t)!} \cdot t\cdot \frac{(\mu_2+[\mu_2]-t-1)!}{\mu_2!([\mu_2]-t)!} \\ \notag
& = \frac{1}{r+1} \cdot \binom{\mu_1+[\mu_1]}{\mu_1} \binom{\mu_2+[\mu_2]}{\mu_2}
\end{align}
in Case II. \par
Let us show this for Case I first. Observe that $tr-\langle \mu_1\rangle = ([\mu_1]+t)r-\mu_1$ and $\mu_1+\mu_2=([\mu_1]+[\mu_2]+1)r$, so we can rewrite the left hand side of equation~\eqref{eq:TechnicalEquationProof02} as 
\begin{align}
& r\cdot (\mu_1+\mu_2)\cdot \sum_{t=1}^{[\mu_2]+1} \frac{(\mu_1+[\mu_1]+t-1)!}{\mu_1!([\mu_1]+t-1)!}  \frac{(\mu_2+[\mu_2]-t)!}{\mu_2!([\mu_2]+1-t)!}
\\ \notag 
& - r\cdot ([\mu_1]+[\mu_2]+1)\cdot \sum_{t=1}^{[\mu_2]+1} \frac{(\mu_1+[\mu_1]+t-1)!}{(\mu_1-1)!([\mu_1]+t)!}  \frac{(\mu_2+[\mu_2]-t)!}{\mu_2!([\mu_2]+1-t)!}
\end{align}
Let us omit the factor $r$ since we have it in the right hand side of equation~\eqref{eq:TechnicalEquationProof02}. Let us multiply the first summand by $\mu_1$ and the second summand by $([\mu_1]+t)$. We get identical sums with the opposite signs. So, this expression divided by $r$ is equal to 
\begin{align}
& \sum_{t=1}^{[\mu_2]+1} \frac{(\mu_1+[\mu_1]+t-1)!}{\mu_1!([\mu_1]+t-1)!}  \frac{(\mu_2+[\mu_2]-t)!}{(\mu_2-1)!([\mu_2]+1-t)!}
\\ \notag 
& - \sum_{t=1}^{[\mu_2]} \frac{(\mu_1+[\mu_1]+t-1)!}{(\mu_1-1)!([\mu_1]+t)!}  \frac{(\mu_2+[\mu_2]-t)!}{\mu_2!([\mu_2]-t)!} \\ \notag
& =: \sum_{t=1}^{[\mu_2]+1} A_t - \sum_{t=1}^{[\mu_2]} B_t
\end{align}
We can reshuffle the summands in this expression in the following way:
\begin{equation}
A_{[\mu_2]+1} - B_{[\mu_2]}+ A_{[\mu_2]} - B_{[\mu_2]-1}+\cdots + A_{2} - B_{1}+A_1
\end{equation}
Now we add up term by term, starting at the left. First we get
\begin{align*}
A_{[\mu_2]+1} - B_{[\mu_2]} & = \binom{\mu_1+[\mu_1]+[\mu_2]}{\mu_1} \binom{\mu_2}{\mu_2} - \binom{\mu_1+[\mu_1]+[\mu_2]-1}{\mu_1-1} \binom{\mu_2}{\mu_2}
\\
&  = \binom{\mu_1+[\mu_1]+[\mu_2]-1}{\mu_1} \binom{\mu_2}{\mu_2}
\end{align*}
Iterating this, get get the following sequence of expressions:
\begin{align}
 A_{[\mu_2]+1}& - B_{[\mu_2]} + A_{[\mu_2]} \\
& = \binom{\mu_1+[\mu_1]+[\mu_2]-1}{\mu_1} \binom{\mu_2}{\mu_2} + 
\binom{\mu_1+[\mu_1]+[\mu_2]-1}{\mu_1} \binom{\mu_2}{\mu_2-1}\\
& = \binom{\mu_1+[\mu_1]+[\mu_2]-1}{\mu_1} \binom{\mu_2+1}{\mu_2}\\
A_{[\mu_2]+1}& - B_{[\mu_2]} + A_{[\mu_2]} - B_{[\mu_2]-1} \\
& = \binom{\mu_1+[\mu_1]+[\mu_2]-1}{\mu_1} \binom{\mu_2+1}{\mu_2}
- \binom{\mu_1+[\mu_1]+[\mu_2]-2}{\mu_1-1} \binom{\mu_2+1}{\mu_2} \\
& = \binom{\mu_1+[\mu_1]+[\mu_2]-2}{\mu_1} \binom{\mu_2+1}{\mu_2}
\end{align}
eventually ending up at
\begin{align}
&  A_{[\mu_2]+1} - B_{[\mu_2]} + \cdots + A_1 = \binom{\mu_1+[\mu_1]}{\mu_1} \binom{\mu_2+[\mu_2]}{\mu_2}
\end{align}
which gives us equation~\eqref{eq:TechnicalEquationProof02}.\par

In Case II, the computation is similar. Observe that $t = ([\mu_1]+t)-\mu_1/r$ and $(\mu_1+\mu_2)/r=[\mu_1]+[\mu_2]$, so we can rewrite the left hand side of equation~\eqref{eq:TechnicalEquationProof03} in the following way:
\begin{align}
& (\mu_1+\mu_2)\cdot \sum_{t=1}^{[\mu_2]} \frac{(\mu_1+[\mu_1]+t-1)!}{\mu_1!([\mu_1]+t-1)!}  \frac{(\mu_2+[\mu_2]-t-1)!}{\mu_2!([\mu_2]-t)!}
\\ \notag 
& - ([\mu_1]+[\mu_2])\cdot \sum_{t=1}^{[\mu_2]} \frac{(\mu_1+[\mu_1]+t-1)!}{(\mu_1-1)!([\mu_1]+t)!}  \frac{(\mu_2+[\mu_2]-t-1)!}{\mu_2!([\mu_2]-t)!}
\end{align}
Again, if we multiply the first summand by $\mu_1$ and the second summand by $([\mu_1]+t)$, this yields identical sums with opposite signs. Cancelling these terms, we get that this expression is equal to 
\begin{align}
& \sum_{t=1}^{[\mu_2]} \binom{\mu_1 + [\mu_1 ] + t -1}{\mu_1} \binom{\mu_2 + [\mu_2 ]-t-1}{\mu_2 -1}
\\ \notag 
& - \sum_{t=1}^{[\mu_2]-1} \binom{\mu_1 + [\mu_1 ] + t -1}{\mu_1-1} \binom{\mu_2 + [\mu_2 ]-t-1}{\mu_2} \\
& =: \sum_{t=1}^{[\mu_2]} A'_t - \sum_{t=1}^{[\mu_2]-1} B'_t
\end{align}
Reshuffling the summands in this expression in the same way as for Case I, we would now get
\begin{equation}
A'_{[\mu_2]} - B'_{[\mu_2]-1}+ A'_{[\mu_2]-1} - B'_{[\mu_2]-2}+\cdots + A'_{2} - B'_{1}+A'_1
\end{equation}
We will calculate this in the same way as before: we start at the right and at the next term one at a time. First we get
\begin{align*}
A'_{[\mu_2]} - B'_{[\mu_2]-1} & = \binom{\mu_1+[\mu_1]+[\mu_2]-1}{\mu_1} \binom{\mu_2}{\mu_2} - \binom{\mu_1+[\mu_1]+[\mu_2]-2}{\mu_1-1} \binom{\mu_2}{\mu_2}
\\
&  = \binom{\mu_1+[\mu_1]+[\mu_2]-2}{\mu_1} \binom{\mu_2}{\mu_2}
\end{align*}
Iterating this, the next few calculations give us the following result:
\begin{align*}
& A'_{[\mu_2]} - B'_{[\mu_2]-1} + A'_{[\mu_2]-1} \\
& = \binom{\mu_1+[\mu_1]+[\mu_2]-2}{\mu_1} \binom{\mu_2}{\mu_2} + 
\binom{\mu_1+[\mu_1]+[\mu_2]-2}{\mu_1} \binom{\mu_2}{\mu_2-1}\\ 
& = \binom{\mu_1+[\mu_1]+[\mu_2]-2}{\mu_1} \binom{\mu_2+1}{\mu_2}\\
& A'_{[\mu_2]} - B'_{[\mu_2]-1} + A'_{[\mu_2]-1} - B'_{[\mu_2]-2} \\
& = \binom{\mu_1+[\mu_1]+[\mu_2]-2}{\mu_1} \binom{\mu_2+1}{\mu_2}
- \binom{\mu_1+[\mu_1]+[\mu_2]-3}{\mu_1-1} \binom{\mu_2+1}{\mu_2} \\
& = \binom{\mu_1+[\mu_1]+[\mu_2]-3}{\mu_1} \binom{\mu_2+1}{\mu_2}
\end{align*}
And finally we get the following result:
\begin{align*}
  A'_{[\mu_2]} - B'_{[\mu_2]-1} + \cdots + A'_1 &= \binom{\mu_1+[\mu_1]}{\mu_1} \binom{\mu_2+[\mu_2]-1}{\mu_2} \\
&= \frac{1}{r+1}\binom{\mu_1+[\mu_1]}{\mu_1} \binom{\mu_2+[\mu_2]}{\mu_2}
\end{align*}
which gives us equation~\eqref{eq:TechnicalEquationProof03}.\par
This way we prove equation~\eqref{eq:2pointIdentityGenFunctLog} is satisfied up to the kernel of the Euler operator. Since neither the left hand side nor the right hand side of equation~\eqref{eq:2pointIdentityGenFunctLog} contain the terms in the kernel of the Euler operator, we see that equation~\eqref{eq:2pointIdentityGenFunctLog} is satisfied, and this completes the proof of the theorem.
\end{proof}



\end{document}